\documentclass[a4paper]{scrartcl}
\usepackage[T1]{fontenc}
\usepackage[english]{babel}
\usepackage[latin1]{inputenc}
\usepackage{amsmath}
\usepackage{amsthm,amssymb}
\usepackage{csquotes}
\usepackage{amssymb}
\usepackage{textcomp}
\usepackage{graphicx}
\usepackage{hyperref}
\usepackage{xcolor}
\usepackage{dsfont}
\usepackage{cancel,ulem}
\newcommand{\erww}[1]{\mathbb{E}\left[{#1}\right]}

\newcommand{\capa}{\operatorname{cap}}
\newcommand{\R}{\mathbb{R}}
\newcommand{\N}{\mathbb{N}}
\newcommand{\erw}{\mathbb{E}}

\newcommand{\eps}{\ensuremath{\varepsilon}}
\newcommand{\W}{\mathcal{W}}
\newcounter{cst}
\newcommand{\ctel}[1]{K_{\refstepcounter{cst}\thecst}}

\newcommand{\diver}{\operatorname{div}}

\newtheorem{theorem}{Theorem}[section]
\newtheorem{definition}[theorem]{Definition}
\newtheorem{lemma}[theorem]{Lemma}
\newtheorem{corollary}[theorem]{Corollary}
\newtheorem{remark}[theorem]{Remark}
\newtheorem{proposition}[theorem]{Proposition}

\usepackage{pdfsync}

\def\dint{\displaystyle\int}

\title{A nonlinear stochastic diffusion-convection equation with reflection}
\author{Niklas Sapountzoglou \thanks{Institute of Mathematics, Clausthal University of Technology, 38678 Clausthal-Zellerfeld, Germany \href{mailto:niklas.sapountzoglou@tu-clausthal.de}{\texttt{niklas.sapountzoglou@tu-clausthal.de}}} \and Yassine Tahraoui \thanks{Scuola Normale Superiore, Piazza dei Cavalieri 7, 56126 Pisa, Italy \href{mailto:yassine.tahraoui@sns.it}{\texttt{yassine.tahraoui@sns.it}}} \and Guy Vallet \thanks{Laboratory of Mathematics and 
Applications of Pau (LMAP) UMR CNRS 5142 \href{mailto:guy.vallet@univ-pau.fr}{\texttt{guy.vallet@univ-pau.fr}}}\and  Aleksandra Zimmermann \thanks{Institute of Mathematics, Clausthal University of Technology, 38678 Clausthal-Zellerfeld, Germany \href{mailto:aleksandra.zimmermann@tu-clausthal.de}{\texttt{aleksandra.zimmermann@tu-clausthal.de}}}}
\date{\today}

\begin{document}
	\maketitle
	\begin{abstract}
		We study a nonlinear, pseudomonotone, stochastic diffusion-convection evolution problem on a bounded spatial domain, in any space dimension, with homogeneous boundary conditions and reflection. The additive noise term is given by a stochastic It\^{o} integral with respect to a Hilbert space valued $Q$-Wiener process. We show existence of a solution to the pseudomonotone stochastic diffusion-convection equation with non-negative initial value as well as the existence of a reflection measure which prevents the solution from taking negative values. In order to show a minimality condition of the measure, we study the properties of quasi everywhere defined representatives of the solution with respect to parabolic capacity.
	\end{abstract}
	\section{Introduction}
\subsection{Motivation of the study and state of the art}
	We want to study the following nonlinear, pseudomonotone diffusion-convection problem with additive colored noise and reflection
		\[	\begin{cases}
			du \ -\operatorname{div}\,a(\cdot,u,\nabla u)\,dt + f(u) \, dt= \Phi \,dW +d\eta \\
			u(0,\cdot)=u_0\geq 0\\
			u\geq	0, \	\eta\geq 0 \	\text{and} \ \int_{Q_T} u \,d\eta=0
		\end{cases}\tag{P}\]
		on a bounded Lipschitz domain of $\mathbb{R}^d$. According to \cite[Section 3]{Par21}, without $\eta$ the sign of the solution $u$ would oscillate randomly. The \textit{reflection measure} $\eta$ prevents the solution $u$ from crossing $0$. The condition $\int_{Q_T} u \,d\eta=0$ says that the pushing away from zero is minimal in the sense that the support of $\eta$ is included in the set where $u$ is zero.
	Problems with reflection of the form (P) are special cases of obstacle problems which appear in the mathematical modeling of many phenomena such as the evolution of damage in a continuous medium  \cite{bauzet2017},  flows in porous media, phase transition as well as in some statistical mechanic problems in physics, financial mathematics, optimal control and biology, etc. We refer, e.g., to  \cite{DuvautLions72,  Haussmann1989,Yassine2020,Zambotti2017} and the references therein for more details. A rich literature exists also on deterministic obstacle problems, where in the deterministic (linear or nonlinear) setting, the problems can be formulated using variational inequalities, see, e.g., \cite{DuvautLions72, GuibeMTV2020,GUIBE2024,MTV19} and their references.    
	\\
	Concerning previous contributions on stochastic PDEs with reflection, without exhaustiveness, let us mention \cite{Haussmann1989} where the authors proved existence and uniqueness of strong solutions to a linear stochastic partial differential equation of parabolic type with $d$-dimensional Brownian motion in one spatial dimension using the framework of abstract stochastic variational inequalities. Then, existence and uniqueness of solutions to the stochastic heat equation with reflection on the spatial interval $[0, 1]$ with Dirichlet boundary conditions has been considered in \cite{Pardoux1992} for the case of additive space-time white noise and in \cite{Pardoux1993} for multiplicative noise.
	In  \cite{Matoussi2010}, the authors showed the existence and uniqueness of the solution to a quasilinear parabolic stochastic obstacle problem with a backward stochastic integral with respect to a one-dimensional Brownian motion using a probabilistic interpretation of the solution. 
	In \cite{Matoussi2014, Matoussi2015}, the authors proved an existence and uniqueness result for quasilinear stochastic PDEs with obstacle using analytical techniques based on parabolic potential theory. 
	Concerning fully nonlinear stochastic obstacle problems in the frame of Sobolev spaces, let us mention 
	\cite{TahraouiVallet2022}, where the authors proved existence and uniqueness of the variational solution  to an obstacle problem governed by a T-monotone nonlinear operator with multiplicative stochastic forcing using an \textit{ad hoc} perturbation of the stochastic term and a penalization of the constraint. Moreover, Lewy-Stampacchia's inequalities, an estimate ensuring the regularity of the reflection measure, have been proved. The extension of these results to the case of stochastic scalar conservation laws with constraints in the framework of entropy solutions has been considered in \cite{BISWAS2023}. \\
	Concerning some qualitative properties and the behavior of solutions to stochastic obstacle problems, let us mention \cite{Zambotti2001,zhang2010} in the quasilinear, one-dimensional case and \cite{Tahraoui2024invariant} on ergodicity and invariant measures in the case of T-monotone operators in any space dimension. Large deviations were considered in \cite{Zhang2009} for stochastic heat equations with reflection in one spatial dimension, in \cite{Matoussi2021LDP} for  quasilinear stochastic obstacle problems and  in \cite{Tahraoui2024LDP} for T-monotone obstacle problems with multiplicative noise.\\
	Recently, stochastic obstacle problems and Lewy-Stampacchia's inequalities for evolution equations with a nonlinear, second order pseudomonotone diffusion-convection operator of Leray-Lions type have been considered in \cite{STVZ2023}, combining the classical penalization method with a singular perturbation by a higher order operator, and Prokhorov and Skorokhod theorems. In this contribution, the existence and uniqueness results obtained by the authors have been limited by the ordered dual assumption, which ensures the regularity of the reflection measure. In particular, the case of with purely additive colored noise and reflection was not included. In this contribution, we focus on exactly this case. To the best of the authors knowledge, the existence of solutions for nonlinear, pseudomonotone,  diffusion-convection evolution equations with additive noise and reflection in any space dimension is an open problem. To show this result, we may apply the classical penalization method and pass to the limit. At the limit, we find a solution $u$ and a reflection measure $\eta$, which is non-negative random Radon measure on the space-time cylinder $Q_T$. It is known that $\eta$ is locally finite on $Q_T$, but it may not be finite. Therefore, we often use localized versions of $\eta$ or of the equation, respectively, i.e., in the proof of the strong convergence of the gradients, and, to obtain a result of narrow convergence. The most challenging part of our work is to give a sense to the minimality condition $\int_{Q_T} u \, d\eta=0$, if $(u,\eta)$ is a solution to (P). In arbitrary space dimension, $u$ may not be continuous and therefore we study fine properties of both, the solution $u$ and the measure $\eta$ to justify the measurability of an appropriate version of $u$ with respect to $\eta$. The necessary fine properties can be provided by parabolic capacity theory in the framework of soft measures and quasi everywhere defined representatives.
	\subsection{Notations and assumptions}\label{Aotsan}
	For $d\in\mathbb{N}=\{1,2,3,4\ldots\}$ let $D\subset \mathbb{R}^d, d\geq	2$  be a bounded domain with Lipschitz boundary, $T>0$ and $Q_T:=(0,T)\times D$. For an open subset $U\subset Q_T$ or $U\subset D$, the space of continuous functions on the closure of $U$, will be denoted by $\mathcal{C}(\overline{U})$ and the space of infinitely differentiable functions with compact support in $U$ will be denoted by $\mathcal{D}(U)$. We will refer to these functions as test functions. For $2\leq p<\infty$, we will write $p':=\frac{p}{p-1}$ and the standard Sobolev space over $D$ will be denoted by $W^{1,p}(D)$. The closure of $\mathcal{D}(D)$ with respect to the natural norm in $W^{1,p}(D)$ will be denoted by $W^{1,p}_0(D)$\footnote{In our situation, it corresponds to the space of elements of $W^{1,p}(D)$ with null trace.} and its topological dual by $W^{-1,p'}(D)$.   Using the standard multi-index notation $\alpha=(\alpha_1,\ldots,\alpha_d)$ with $\alpha_i\geq 0$ an integer (see, e.g., \cite[Sec. 9.1]{Brezis}), for $k\in\mathbb{N}$, we introduce the Hilbert spaces 
	\[H^k(D):=\left\{v\in L^2(D) \ | \ D^{\alpha}v \in L^2(D) \ \forall \alpha \ \text{with} \ |\alpha|\leq k \right\}, \]
	where $D^{\alpha} v$ are the spatial derivatives of $v$ in the sense of distributions and $H_0^k(D)$ denotes the closure of $\mathcal{D}({D})$ in $H^k(D)$ with respect to the $\Vert \cdot \Vert_{H^k(D)}$-norm. For more details on these function spaces, we refer to \cite[Sec. 9]{Brezis}.\\
	For a function $v:Q_T\rightarrow \mathbb{R}$ we define its positive part by $u^+:=\max\{u,0\}$ and its negative part by $u^{-}:=-\min\{u,0\}$.\\
	
	Let $(\Omega,\mathcal{F},\mathds{P}, (\mathcal{F}_t)_{t\geq 0}, (W_t)_{t \geq 0})$ be a stochastic basis with a complete, right continuous filtration $(\mathcal{F}_t)_{t\geq 0}$ and with a $Q$-Wiener process $(W_t)_{t\geq 0}$ with respect to $(\mathcal{F}_t)_{t\geq 0}$ taking values in $L^2(D)$\footnote{More precisely, let $((\beta^k_t)_{t\geq 0})_k$ be a sequence of independent, real-valued Wiener processes with respect to $(\mathcal{F}_t)_{t\geq 0}$, $U\subset L^2(D)$ be a separable Hilbert space and $Q:U\rightarrow U$ be a non-negative, symmetric trace class operator with $Q^{1/2}(U)=L^2(D)$. If $(e_k)_k$ is an orthonormal basis of $U$ made of eigenvectors of $Q$ with corresponding eigenvalues $(\lambda_k)_k\subset [0,\infty)$, then
			\[W(t):=\sum_{k=1}^{\infty}\sqrt{\lambda_k} e_k\beta^k_t,\ t\geq 0\] }. In this contribution, we want to study a nonlinear stochastic diffusion-convection equation with reflection. The problem of our interest may formally be written as:	find a pair	$(u,\eta)$	solution to
	\begin{align}\label{220822_01}
		\begin{cases}
			du \ -\operatorname{div}\,a(\cdot,u,\nabla u)\,dt + f(u) \, dt= \Phi \,dW +d\eta \\
			u(0,\cdot)=u_0\geq 0\\
			u\geq	0, \ \eta\geq 0 \	\text{and} \ \int_{Q_T} u \,d\eta=0.
		\end{cases}
	\end{align}
	We	refer	to	Theorem	\ref{main theorem}	for the precise notion of solution to	\eqref{220822_01}.\\
	In our setting, $f: \R \to \R$ is Lipschitz continuous with Lipschitz constant $L_f\geq 0$ such that $f(0)=0$ and $a:D\times\mathbb{R}^{d+1} \to \mathbb{R}^d$ is a Carath\'eodory function. Suppose that there exist $p\geq 2$, $\kappa\in L^1(D)$, non-negative functions $g,h\in L^{p'}(D)$ and constants $C_1>0$, $C_2\geq 0$, $C_3\geq 0$, $C_4\geq 0$ such that the following conditions hold:

	\begin{itemize}
		\item[$(A1)$] Monotonicity: $(a(x,\lambda,\xi)-a(x,\lambda,\zeta))\cdot(\xi-\zeta)> 0$
		for all $\lambda\in\mathbb{R}$, $\xi,\zeta\in\mathbb{R}^d$ with $\xi \neq \zeta$ and almost every $x\in D$.
		\item[$(A2)$] Coercivity and growth: 
		\[a(x,\lambda,\xi)\cdot \xi\geq \kappa(x)+C_1|\xi|^{p},\]
		\[|a(x,\lambda,\xi)|\leq C_2|\xi|^{p-1}+C_3|\lambda|^{p-1}+g(x)\]
		for all $\lambda\in\mathbb{R}$, $\xi\in\mathbb{R}^d$ and almost every $x\in D$.
		\item[$(A3)$] Lipschitz regularity: 
		\[|a(x,\lambda_1,\xi)-a(x,\lambda_2,\xi)|\leq \Big(C_4|\xi|^{p-1}+h(x)\Big)|\lambda_1-\lambda_2|\]
		for all $\lambda_1$, $\lambda_2\in \mathbb{R}$, for all $\xi\in\mathbb{R}^d$ and almost all $x\in D$.
	\end{itemize} 
	
	\begin{remark}
	Any Carath\'eodory function $a(x,\lambda,\xi)$ satisfying Assumptions $(A1)$-$(A3)$ induces a pseudomonotone, nonlinear Leray-Lions operator $u\mapsto -\operatorname{div}\, a(\cdot,u,\nabla u)$ defined on $W^{1,p}_0(D)$ with values in $W^{-1,p'}(D)$.
	Please note that $a(x, \lambda, \xi) = |\xi|^{p-2} \xi + F(\lambda)$, where $F:\mathbb{R}\rightarrow \mathbb{R}^d$ is Lipschitz continuous with Lipschitz constant $L_F\geq 0$ and $F(0)=0$, is a suitable choice, since $p \geq 2$. Hence, the $p$-Laplace operator with a first-order convection term given by $-\operatorname{div}\, a(\cdot,u,\nabla u) = -\operatorname{div}\, (|\nabla u|^{p-2} \nabla u + F(u))$ fits into our setting.
			\end{remark}
	
	The stochastic integral on the right-hand side of \eqref{220822_01} is understood in the sense of It\^{o}. In the following, we will denote the space of Hilbert-Schmidt operators from a separable Hilbert space $K$ to a separable Hilbert space $H$ by $HS(K;H)$. We assume the integrand $\Phi$ to be progressively measurable in $L^p(\Omega\times (0,T);HS(L^2(D);H^k(D)\cap H^1_0(D)))$ for $k\in\mathbb{N}$ such that $k>\max(\frac{d}{2}, \frac{2+d}{2} - \frac{d}{p})$.
	According to \cite[Theorem 1.4.4.1 and Theorem 1.4.3.1]{Gr}, these assumptions on $k$ ensure that $H^k(D)\cap H^1_0(D) \hookrightarrow \mathcal{C}(\overline{D})\cap W^{1,p}(D)\cap H^1_0(D)\subset \mathcal{C}(\overline{D})\cap W^{1,p}_0(D)$. Consequently, by \cite[Lemma 2.1]{FG95} the stochastic integral $\int_0^{\cdot} \Phi(s)\,dW_s$ is in $ L^p(\Omega;\mathcal{C}([0,T];H^k(D)\cap H^1_0(D) ))$  and therefore in $L^2(\Omega;\mathcal{C}([0,T];\mathcal{C}(\overline{D}))) \cap L^p(\Omega;L^p(0,T;W_0^{1,p}(D)))$.
	

\subsection{Outline}
Our work is organized as follows: In Section \ref{s2} we propose a definition of a solution to \eqref{220822_01} and formulate our main theorem on existence of solutions, namely Theorem \ref{main theorem}. Sections \ref{s3} to \ref{s5} are devoted to the proof of Theorem \ref{main theorem}. In Section \ref{s3}, we introduce the penalization procedure, prove \textit{a-priori} estimates, deduce convergence results and pass to the limit in the equation. Section \ref{s4} is devoted to the a.e. convergence of the gradients and the characterization of the reflection measure $\eta$ and its localized versions. In Section \ref{s5}, we study pointwise properties of the solution $u$ as well as the properties $\eta$ using parabolic capacity theory. In Section \ref{section-p-leq-2}, we consider the case when the parameter $p$ introduced within the Assumptions $(A1)$-$(A3)$ is allowed to be less than $2$. Necessary results from parabolic capacity theory are proved in the Appendix (see Section \ref{A}).

\section{Main results}\label{s2}
 \subsection{Complements of measure theory}
For any topological space $X$, the Borel-sigma field on $X$ will be denoted by $\mathcal{B}(X)$. In the following, we denote the vector space of continuous functions $\psi:Q_T\rightarrow\mathbb{R}$ with compact support by $\mathcal{C}_c(Q_T)$ and endow this space with its usual topology. Its dual space will be denoted by $\mathcal{C}_c(Q_T)^{\ast}$, which is identified with the set of Radon measures. An element $T\in \mathcal{C}_c(Q_T)^{\ast}$ is called \textit{non-negative (Radon measure)}, iff $T(\psi)\geq 0$ for all $\psi\in \mathcal{C}_c(Q_T)$ such that $\psi\geq 0$. For more details, see, e.g., \cite{M74,Rudin}. 	
\subsection{Definitions and main theorem}
In the following, whenever an object obviously depends on $\omega \in \Omega$, we don't write down this dependence on $\omega$ unless it is important to understand the arguments.
	
	\begin{definition}
		A random Radon measure is a family $(\eta_{\omega})_{\omega\in\Omega}$ of Radon measures. For simplicity, we will often omit the variable $\omega\in\Omega$ and write $\eta=(\eta_{\omega})_{\omega\in\Omega}$.
	\end{definition}
	\begin{definition}
		A random Radon measure $\eta$ on $Q_T$ is called weak-$\ast$ adapted iff, for any $\psi\in \mathcal{C}_c(Q_T)$ and  any $t\in(0,T)$, the mapping 
		\begin{align*}
			\Omega \ni \omega \mapsto \int_{(0,t]\times D} \psi(s,x)  \ d\eta_{\omega}(s,x) \quad \text{is $\mathcal{F}_t$-measurable.}
		\end{align*}	
	\end{definition}

\begin{definition}
		For $2\leq p<\infty$ we define the space
		\[\mathcal{W}:=\{v\in L^p(0,T;W^{1,p}_0(D)) \ | \ \partial_t v \in L^{p'}(0,T;W^{-1,p'}(D))\}.\]
		It is a separable reflexive Banach space endowed with its natural norm $\Vert \cdot \Vert_{\mathcal{W}}$ (see Section \ref{A}: Appendix for more details). Its dual space will be denoted by $\mathcal{W}'$.
	\end{definition}
\begin{definition}\label{Solution}
A pair $(u,\eta)$ is a solution to the obstacle problem \eqref{220822_01}, if and only if 
\begin{itemize}
\item[$i.)$] $u \geq 0$ a.e. in $\Omega\times Q_T$, 
\[u \in L_w^2(\Omega,L^\infty(0,T;L^2(D)))\cap L^p(\Omega; L^p(0,T; W_0^{1,p}(D)))\] 
\item[$ii.)$] There exists a right-continuous in time, adapted representative of $u$ with values in $L^{2}(D)$ satisfying $u(t=0^+)=u_0$.
\item[$iii.)$] There exists $\eta \in L^{p^\prime}(\Omega, \mathcal{W}^\prime)$ that is also a weak-$\ast$ adapted random non-negative Radon measure on $Q_T$ with, $\mathds{P}$-a.s. 
\begin{align*}
\partial_t\left(u-\int_0^{\cdot} \Phi \, dW\right) - \operatorname{div} a(\cdot, u , \nabla u) + f(u) = \eta \quad \text{in }\mathcal{W}^\prime.
\end{align*}
\item[$iv.)$] $\mathds{P}$-a.s. in $\Omega$ there exists a quasi everywhere defined representative of $u$ in $Q_T$, non-negative quasi everywhere in $Q_T$, such that $\dint_{Q_T} u\ d\eta =0$.
\end{itemize}
\end{definition}	
\begin{remark}
Note that by Proposition \ref{cadlag}, the solution $u$ in the sense of Definition \ref{Solution} is also a.s. weakly-left-limited in $L^2(D)$, thus c\`adl\' ag with values in  $W^{-1,p'}(D)$.
\end{remark}
	\begin{theorem}[Existence of solutions]\label{main theorem}
Let the assumptions in Section \ref{Aotsan} be satisfied and $u_0 \in L^2(\Omega \times D)$ be $\mathcal{F}_0$-measurable with $u_0 \geq 0$ a.e. in $\Omega \times D$. 
There there exists a solution $(u,\eta)$ in the sense of Definition \ref{Solution}.
	\end{theorem}
	
	\begin{remark}	Let $\psi:D\to \mathbb{R}$ be an obstacle function such that $\psi\in W^{1,p}_0(D)$. Then, for a smaller class of operators, our analysis yields the existence of a solution to \eqref{220822_01} with $u\geq \psi$. More precisely, assume moreover that the operator is independent of the second argument, i.e., $a(x,\lambda,\xi)=a(x,\xi)$ and is strongly monotone, namely there exists $\beta>0$ such that
			\begin{align}\label{strong-mono}
				(a(x,\xi)-a(x,\zeta)) \cdot (\xi-\zeta)\geq  \beta \vert \xi-\zeta \vert^p
			\end{align}
			for all $\xi,\zeta\in \mathbb{R}^d$ with $\xi \neq \zeta$.
			An example of an operator satisfying $(A1)$-$(A3)$ and \eqref{strong-mono} is given by the p-Laplace operator $ - \operatorname{div} a(\cdot, \nabla u) = - \operatorname{div} (\vert \nabla u \vert^{p-2} \nabla u)$.
			 Indeed, setting $v:=u-\psi \geq 0$ and considering the modified operator $\widetilde{a}(x,\nabla v)=a(x,\nabla (v+\psi))$, it is not difficult to check that $\widetilde{a}$ satisfies $(A1)$-$(A3)$. 
	\end{remark}
	\section{Approximation procedure, estimates and limit equation}\label{s3}
	\begin{remark}\label{Remark 110724_01}
		A slight modification of the results in \cite{VZ21} yields that there exists a unique probabilistically strong solution to the stochastic partial differential equation
		\begin{align}\label{1}
			u(t)= u_0+\int_0^t \operatorname{div}a(\cdot, u , \nabla u) \, ds - \int_0^t f(u) \, ds  + \int_0^t \Phi \, dW,
		\end{align}
		i.e., there exists a unique predictable
		$u \in L^2(\Omega;\mathcal{C}([0,T]; L^2(D)))\cap L^p(\Omega; L^p(0,T; W_0^{1,p}(D)))$, $u(0, \cdot)=u_0$ in $L^2(\Omega\times D)$
		and $\mathds{P}$-a.s. equation \eqref{1} holds in $L^2(D)$, for all $t \in [0,T]$. Moreover let us remark that $\mathds{P}$-a.s.
		\begin{align*}
			\partial_t\Big[u-\int_0^{\cdot} \Phi \, dW\Big] - \operatorname{div} a(\cdot, u , \nabla u)  + f(u) = 0	\quad	\text{ in $L^{p'}(0,T; W^{-1,p'}(D)).$}
		\end{align*}
	\end{remark}
	\begin{remark}\label{241211_rem01}
		The assumptions in Section \ref{Aotsan} yield that 
		\[u-\int_0^{\cdot} \Phi \, dW \in L^p(\Omega; L^p(0,T; W_0^{1,p}(D)))\] 
		and 
		\[\partial_t \Big[u-\int_0^{\cdot} \Phi \, dW \Big] \in L^{p'}(\Omega; L^{p'}(0,T; W^{-1,p'}(D))),\]
		hence $u-\int_0^{\cdot} \Phi \, dW \in L^{p'}(\Omega, \mathcal{W})$. Moreover, since $\mathds{P}$-a.s. $\int_0^{\cdot} \Phi \, dW \in \mathcal{C}(\overline{Q_T})$, Lemma \ref{220822_lem01} yields that, $\mathds{P}$-a.s., $u$ is $\capa_p$-quasi continuous. Moreover, by Lemma \ref{241211_lem01}, the $\capa_p$-quasi continuous representative of $u$ coincides with the representative of $u\in \mathcal{C}([0,T];L^2(D))$ $\mathds{P}$-a.s. In particular, for the $\capa_p$-quasi continuous representative of $u$ we have, $\mathds{P}$-a.s., $u(0)=u_0$ a.e. in $D$.
	\end{remark}
	First of all we want to show a comparison principle with respect to the Lipschitz continuous perturbation $f$ and the initial value $u_0$.
	
	\begin{lemma}\label{Lemma 1}
		Let $u$ be a strong solution to \eqref{1} with Lipschitz continuous reaction term $f$ and initial value $u_0 \in L^2(\Omega \times D)$. Moreover, let $v$ be a strong solution to \eqref{1} with Lipschitz continuous reaction term $g$ and initial value $v_0\in L^2(\Omega \times D)$, where $f \geq g$ and $u_0 \leq v_0$ a.e. in $\Omega \times D$.
Then, there exists a measurable set $\Omega_{f,g,u,v}\in\mathcal{F}$ of full measure 
		such that for all $\omega\in \Omega_{f,g,u,v}$:
\begin{itemize}
		\item [$i.)$] For all $t \in [0,T]$,  $u(t) \leq v(t)$ a.e. in $D$. 
\item [$ii.)$] $u \leq v$ quasi everywhere in $Q_T$ (see Section \ref{A}: Appendix).
\end{itemize}
	\end{lemma}
	\begin{proof}
		For $\delta>0$ we consider the sequence $(\eta_{\delta})_{\delta >0}$ of real valued functions which is a smooth approximation of the positive part, denoted by $\operatorname{id}^+$ in the following, i.e.,
		\begin{align*}
			\eta_{\delta}(r):= \begin{cases}
				0, \, &r <0, \\
				-\frac{1}{2\delta^3}r^4 + \frac{1}{\delta^2}r^3, \, &0 \leq r \leq \delta, \\
				r- \frac{\delta}{2}, \, &r > \delta.
			\end{cases}
		\end{align*}
		Moreover, we have, for all $r \in \R$:
		\begin{align*}
			\begin{array}{lll}
				\bullet\quad \eta_{\delta} \in C^2(\mathbb{R})  
				&\bullet\quad 0 \leq \eta_{\delta}(r) \leq r^+   
				&\bullet\quad \eta_{\delta}(r) \to \operatorname{id}^+(r)\text{ as }\delta \to 0^+  
				\\[0.1cm] 
				\bullet\quad   0 \leq \eta_{\delta}'(r) \leq 1
				&\bullet\quad  0 \leq \eta_{\delta}''(r) \leq \frac{3}{2\delta}  
				&\bullet\quad  \operatorname{supp} \eta_{\delta}'' \subset [0, \delta].
			\end{array}
		\end{align*}
		There exists a measurable set $\Omega_{f,g,u,v} \subset \Omega$, of full measure, \textit{a priori} depending on $f,g,u,v$, such that for each $\omega\in\Omega_{f,g,u,v}$, we	have $u,v \in \mathcal{C}([0,T];L^2(D))$ and 
		\begin{align}\label{eq110724_01}
			\partial_t[u-v] - \operatorname{div} a(\cdot, u , \nabla u)-a(\cdot, v , \nabla v) + f(u) -g(v) = 0
		\end{align}
		in $L^{p'}(0,T; W^{-1,p'}(D))$. Since $u-v= u - \int_0^{\cdot} \Phi \, dW - (v - \int_0^{\cdot} \Phi \, dW) \in \mathcal{W}$,	we can test \eqref{eq110724_01} with $\eta_{\delta}'(u-v)$ and apply integration by parts in time to obtain
		\begin{align*}
			I_1 -I_2 + I_3 + I_4=0,
		\end{align*}
		where
		\begin{align*}
			I_1 &=\int_D \eta_{\delta} (u(t) - v(t)) \, dx  \to \int_D (u(t) - v(t))^+ \, dx ~\textnormal{as} ~ \delta \to 0^+,\\
			I_2 &=\int_D \eta_{\delta} (u_0 - v_0) \, dx =0, 
		\end{align*}
		\begin{align*}
			I_3 =&\int_0^t \int_D \eta_{\delta}'' (u - v) [a(x,u,\nabla u) - a(x,v, \nabla v)] \cdot \nabla (u-v)  \, dx \, ds \\
			=&\int_0^t \int_D \eta_{\delta}'' (u - v) [a(x,u,\nabla u) - a(x,u, \nabla v)] \cdot \nabla (u-v)  \, dx \, ds \\
			&+\int_0^t \int_D \eta_{\delta}'' (u - v) [a(x,u,\nabla v) - a(x,v, \nabla v)] \cdot \nabla (u-v)  \, dx \, ds \\
			\geq& \int_0^t \int_D \eta_{\delta}'' (u - v) [a(x,u,\nabla v) - a(x,v, \nabla v)] \cdot \nabla (u-v)  \, dx \, ds \\
			\geq& -\int_0^t \int_D \eta_{\delta}''(u-v) \Big[|C_4 |\nabla v|^{p-1} + h(x)|\Big] |u-v| |\nabla(u-v)|  \to 0 ~\textnormal{as} ~ \delta \to 0^+. 
		\end{align*}
		\begin{align*}
			I_4 &= \int_0^t \int_D  (f(u) - g(v)) \eta_{\delta}'(u-v) \, dx \, ds 
			\\
			&=\int_0^t \int_D  (f(u) - f(v) + f(v) - g(v)) \eta_{\delta}'(u-v) \, dx \, ds 
			\\
			&\geq \int_0^t \int_D (f(u) - f(v)) \eta_{\delta}'(u-v) \, dx \, ds 
			\\
			&\geq - L_f\int_0^t \int_D |u-v| \eta_{\delta}'(u-v) \, dx \, ds
			= - L_f\int_0^t \int_{D\cap\{u \geq v\}} (u-v) \eta_{\delta}'(u-v) \, dx \, ds
			\\
			&\to 
			- L_f\int_0^t \int_D (u-v)^+ \, dx \, ds ~\textnormal{as} ~ \delta \to 0^+.
		\end{align*}
		Therefore, for any $\omega\in  \Omega_{f,g,u,v} $ we have
		\begin{align*}
			\int_D (u(t) - v(t))^+ \, dx \leq L_f\int_0^t \int_D (u-v)^+ \, dx \, ds
		\end{align*}
		for all $t \in [0,T]$. Gronwall's lemma yields
		\begin{align}\label{220421_01}
			\int_D (u(t) - v(t))^+ \, dx =0
		\end{align}
		for all $t \in [0,T]$. Hence, for any $\omega\in\Omega_{f,g,u,v}$  and all $t\in[0,T]$, $(u(t)-v(t))^{+}=0$ a.e. in $D$ and therefore $u(t) \leq v(t)$ a.e. in $D$. In particular, for any $\omega\in \Omega_{f,g,u,v}$, $u \leq v$ a.e. in $Q_T$ and Corollary \ref{LemPositiveQuasiCont} yields $u \leq v$ q.e. on $Q_T$ for all $\omega\in\Omega_{f,g,u,v}$, where the quasi-exceptional set in $Q_T$ may depend on $\omega\in\Omega_{f,g,u,v}$ .
		
	\end{proof}
	Now, for $n \in \N$ we set $f_n(r):= f(r) - n r^-$. Then, there exists a unique strong solution to equation \eqref{1} with $f$ replaced by $f_n$, \textit{i.e.}, there exists a unique predictable process $u_n \in L^2(\Omega; \mathcal{C}([0,T]; L^2(D)))\cap L^p(\Omega; L^p(0,T; W_0^{1,p}(D)))$ such that $u_n(0, \cdot)=u_0$ in $L^2(\Omega\times D)$ and $\mathds{P}$-a.s. in $\Omega$
	\begin{align}\label{2}
		u_n(t)= u_0 +\int_0^t \operatorname{div} a(\cdot, u_n , \nabla u_n) \, ds - \int_0^t f(u_n) \, ds + n \int_0^t u_n^- \, ds + \int_0^t \Phi \, dW
	\end{align}
	holds true in $L^2(D)$, for all $t \in [0,T]$. 
	Moreover, $\mathds{P}$-a.s. in $\Omega$, $u_n$ is a solution  in $L^{p'}(0,T; W^{-1,p'}(D))$ to the equation
	\begin{align}\label{3}
		\partial_t (u_n - \int_0^{\cdot} \Phi \, dW) - \operatorname{div} a(\cdot, u_n , \nabla u_n) + f(u_n) - nu_n^- =0.
	\end{align}

	\subsection{Estimates in expectation}\label{subsection-estimate-exp}
	\begin{lemma}\label{110724_lem1}
		Let $u_n$ be the unique strong solution to \eqref{2}. Then, for all $t\in [0,T]$, 
		\begin{align}\label{231019_06}
			\begin{aligned}
				&\frac{1}{2}\erww{\Vert u_n(t)\Vert_2^2}+\erww{\int_0^t\int_D\kappa(x)+C_1|\nabla u_n(s)|^p\,dx\,ds}+\erww{\int_0^t\Vert\sqrt{n}u_n^-(s)\Vert_2^2\,ds}\\
				&\leq \frac{1}{2}\erww{\Vert u_0\Vert_2^2}+\frac{1}{2}\erww{\int_0^t\Vert \Phi(s)\Vert^2_{\operatorname{HS}}\,ds}+L_f\erww{\int_0^t\Vert u_n(s)\Vert_2^2\,ds}.
			\end{aligned}
		\end{align}	
	\end{lemma}
	\begin{proof}
		Applying  It\^{o}'s formula with $\frac{1}{2} \Vert \cdot \Vert_{L^2(D)}^2$ to equality \eqref{2} yields $\mathds{P}$-a.s. in $\Omega$
		\begin{align}\label{4}
			\begin{aligned}
				&\frac{1}{2} \int_D u_n(t)^2 \, dx + \int_0^t \int_D a(x, u_n, \nabla u_n) \cdot \nabla u_n \, dx \, ds \\
				+&\int_0^t \int_D f(u_n)u_n \, dx \, ds+n \int_0^t \int_D |u_n^-|^2 \, dx \, ds \\
				= &\frac{1}{2} \int_D u_0^2 \, dx + \int_0^t (u_n , \Phi \, dW)_2 + \frac{1}{2}\int_0^t \Vert \Phi(s)\Vert^2_{\operatorname{HS}} \,ds 
			\end{aligned}
		\end{align}
		for all $t \in [0,T]$, for all $n\in \mathbb{N}$. Taking the expectation in \eqref{4}, using (A2) yields the result.
	\end{proof}
	
	From Lemma \ref{110724_lem1} we immediately have the following result:
	
	\begin{lemma}\label{231019_lem1}
		There exist constants $K_1,K_2,K_3,K_4 >0$, not depending on $n \in \N$ such that
		\begin{align}\label{240806_01}
			\erww{\sup_{t\in[0,T]}\Vert u_n(t)\Vert_2^2}\leq K_1,
		\end{align}
		
		\begin{align}\label{240806_02}
			\erww{\int_0^T\Vert \nabla u_n(s)\Vert_p^p\,dt}\leq K_2,
		\end{align}
		
		\begin{align}\label{240806_03}
			\erww{\int_0^T\int_D |a(x, u_n,\nabla u_n)|^{p'}\,dx\,dt}\leq K_3,
		\end{align}
		
		\begin{align}\label{240806_04}
			\erww{\int_0^T\left\Vert\sqrt{n} u_n^{-}(s)\right\Vert_2^2\,ds}\leq K_4.
		\end{align}
	\end{lemma}
	\begin{proof}
		Discarding non-negative terms and applying Gronwall's lemma in \eqref{231019_06}, we get that 
		\[\sup_{t\in[0,T]}\erww{\Vert u_n(t)\Vert_2^2}\] 
		is uniformly bounded with respect to $n\in\mathbb{N}$. Then, \eqref{240806_02}-\eqref{240806_04} follow directly from \eqref{231019_06} and this bound. The stronger result \eqref{240806_01} may be obtained by using Burkholder's inequality.
	\end{proof}
	
	\subsection{Pathwise properties}\label{subsection-pathwise-esti}
	\begin{lemma}\label{220805_lem01}
		For $n\in\mathbb{N}$, let $u_n$ be the solution to \eqref{2}. There exists a full measure set $\widetilde{\Omega}\in\mathcal{F}$ such that for all $\omega\in\widetilde{\Omega}$:  for all $t\in [0,T]$,  we have $u_n(t)\leq u_m(t)$ a.e. in $D$ for any $m,n\in\mathbb{N}$ with $m\geq n$;  and $u_n\leq u_m$ q.e. in $Q_T$  for the same $n$ and $m$.
	\end{lemma}
	\begin{proof}
		For fixed $m,n,\in\mathbb{N}$ such that $m\geq n$, let $\Omega_{m,n}$ be defined according to Lemma
		\ref{Lemma 1}, \textit{i.e.}, $\mathds{P}(\Omega_{m,n})=1$ and such that for any $\omega\in\Omega_{m,n}$ we have: for all $t \in [0,T]$, $u_n(t) \leq u_m(t)$ a.e. in $D$ and $u_n \leq u_m$ q.e. in $Q_T$. 
Let us set $\widetilde{\Omega}=\bigcap_{n\in\mathbb{N}}\bigcap_{m\geq n}\Omega_{m,n}$. Then, $\mathds{P}(\Omega_{m,n}^C)=\mathds{P}(\bigcup_{n\in\mathbb{N}}\bigcup_{m\geq n}\Omega_{m,n}^C)=0$ since $\mathds{P}(\Omega_{m,n}^C)=0$ for any $m,n \in \N$.
		Consequently, for any $\omega\in\widetilde{\Omega}$, for all $n,m\in\mathbb{N}$ with $m\geq n$ we have, for all $t\in[0,T]$, $u_n(t)\leq u_m(t)$ a.e. in $D$ and $u_n\leq u_m$ q.e. in $Q_T$.
	\end{proof}
\begin{lemma}\label{240806_lem1}
		There exists a subset of $\Omega$ of full measure, still denoted $\widetilde{\Omega}$, and  a predictable $u\in L^2(\Omega\times(0,T);L^2(D))$ such that, for any $\omega$ in $\widetilde{\Omega}$ and when $n\rightarrow\infty$ for a not relabeled subsequence,
		\begin{itemize}
			\item[$i.)$] 
			$u_n\rightarrow u$ in $L^2(\Omega;L^2(Q_T))$ \ and \ $u_n(\omega)\rightarrow u(\omega)$ in $L^2(Q_T)$.
			\item[$ii.)$] 
			There exists a quasi everywhere defined representative $\bar{u}(\omega)$ of $u(\omega)$ in $Q_T$ such that \linebreak $u_n(\omega,t,x)\rightarrow \bar{u}(\omega,t,x)$ non-decreasingly q.e. in $Q_T$
		\end{itemize}
	\end{lemma}
	\begin{proof}
		Repeating the arguments in the proof of Lemma \ref{Lemma 1} with $u=u_n$, $v=u_{n+1}$ and taking expectation in \eqref{220421_01}, we obtain $u_n\leq u_{n+1}$ $d\mathds{P}\otimes dt\otimes dx$-a.e. in $\Omega\times (0,T)\times D$. Therefore we may define the measurable function $u:\Omega\times [0,T]\times D\rightarrow\mathbb{R}$ by setting
		\[u(\omega,t,x)=\sup_{n\in\mathbb{N}}u_n(\omega,t,x) \]
		$d\mathds{P}\otimes dt\otimes dx$-a.e. in $\Omega\times (0,T)\times D$ and, e.g., equal to zero on the exceptional set. 
		\\
		Using Lemma \ref{231019_lem1}, \eqref{220822_01} and Fatou's lemma, it follows that $u\in L^2(\Omega\times (0,T)\times D)$. Then, since one has
		\[|u_n(\omega,t,x)|\leq |u_1(\omega,t,x)|+|u(\omega,t,x)|\]
		for all $n\in\mathbb{N}$, $d\mathds{P}\otimes dt\otimes dx$-a.e. in $\Omega\times (0,T)\times D$ one concludes by Lebesgue's dominated convergence theorem that $u_n$ converges to $u$ in $L^2(\Omega\times (0,T)\times D)$ and also in $L^2(\Omega;L^2(Q_T))$ and $L^2(\Omega\times (0,T);L^2(D))$ since $L^2(\Omega\times (0,T)\times D)$ is isomorphic to the latter (see, e.g., \cite[Proposition 1.2.24, p.25]{HNVW16}). 
		\\ Passing to a subsequence if necessary, the converse of Lebesgue's dominated convergence theorem yields the  convergence of $u_n(\omega)$ to $u(\omega)$ in $L^2(Q_T)$, $\omega-\mathds{P}$-a.s. in $\Omega$. 
		\\
		For the same subsequence, one may also consider that $u_n(\omega,t)$ converges to $u(\omega,t)$ in $L^2(D)$, $(\omega,t)-d\mathds{P}\otimes dt$-a.e. in $\Omega\times (0,T)$, with respect to the predictable $\sigma$-field $\mathcal{P}_T$. 
Therefore $i.)$  holds true.
		\\ 
		Now, from $i.)$ it follows that there exists $N\in\mathcal{F}$ with $\mathds{P}(N)=0$ such that, up to a subsequence, $u_n(\omega)\rightarrow u(\omega)$ in $L^2(Q_T)$ as $n\rightarrow\infty$ for all $\omega\in\Omega\setminus N$. 
		\\
		Choosing $\widetilde{\Omega}$ according to Lemma \ref{220805_lem01}, for all $\omega\in \bar{\Omega}:=(\widetilde{\Omega}^C\cup N)^C$ it follows that
		\begin{align*}
			u_n(\omega)\leq u_{n+1}(\omega) \quad \text{ q.e. in $Q_T$ and for all $n$.} 
		\end{align*}
		Setting for $(t,x)$ q.e. in $Q_T$
		\[\bar{u}(\omega):=\sup_{n\in\mathbb{N}}\operatorname{quasi \ ess}u_n(\omega,t,x)=\lim_{n\rightarrow\infty}u_n(\omega,t,x),\] 
		we find that $\bar{u}(\omega)=u(\omega)$ $dt\otimes dx$-a.e. in $Q_T$. In this manner, for any $\omega\in\bar{\Omega}$, $u(\omega)\in L^2(Q_T)$ admits a representative which is defined q.e. in $Q_T$ such that
		\[\bar{u}(\omega)=\sup_{n\in\mathbb{N}}\operatorname{quasi \ ess}u_n(\omega,t,x)=\lim_{n\rightarrow\infty}u_n(\omega)\]
		q.e. in $Q_T$.  
	\end{proof}	
	\begin{lemma}\label{110724_rem1}
		Let $u$ be introduced as in Lemma \ref{240806_lem1}. There exists a new subset of $\Omega$ of full measure, still denoted $\widetilde{\Omega}$, and a not relabeled subsequence if necessary, such that,  for all $(\omega,t)\in\widetilde{\Omega}\times [0,T]$,  
		\begin{align*}
			u(\omega,t)=\lim_{n\rightarrow\infty}u_n(\omega,t) \text{ in $L^2(D)$.}
		\end{align*}
	\end{lemma}
	\begin{proof}
		By \cite[Theorem 2.15]{DPP} and an extension of $u_n$ to the interval $(-T,2T)$ as in Lemma \ref{Lemma 261124_04} of the Appendix, for any $t \in [0,T]$,  $\widetilde{\capa}_p(\{t\}\times A)=0$ if and only if $A\subset D$ is Borel measurable with measure zero. Thus, for any $(\omega,t)\in \widetilde{\Omega}\times [0,T]$, $u_n(\omega,t) \to \overline{u}(\omega,t)$ a.e. in $D$ non-decreasingly, where $\overline{u}$ is the chosen representative of $u$ introduced in Lemma \ref{240806_lem1}. 
		Let us prove that this convergence holds also in $L^2(D)$.
		To this end we show that
		\[\mathds{P}\left[\bigcap_{L\geq 1}\bigcup_{N\in\mathbb{N}}\bigcap_{n\geq N} \left\{\omega\in\Omega  :  \sup_{t\in[0,T]} \Vert u_n(\omega,t,\cdot)\Vert_2^2>L\right\}\right]=0.\]
		Indeed: For any $L\geq 1$, $n\in\mathbb{N}$, we set $A_n^L:= \left\{\omega\in\Omega  :  \sup_{t\in[0,T]} \Vert u_n(\omega,t,\cdot)\Vert_2^2>L\right\}$. Recalling that
		\[\bigcap_{n\geq N}A_n^L\uparrow \liminf_{N\rightarrow\infty}A_N^L:=\bigcup_{N=1}\bigcap_{n\geq N}A_N^L\]
		we have
		\begin{align*}
			\mathds{P}\left[\bigcup_{N=1}\bigcap_{n\geq N}A_n^L\right]=\lim_{N\rightarrow\infty}\mathds{P}\left[\bigcap_{n\geq N}A_n^L\right]=\mathds{P}\left[\liminf_{N\rightarrow\infty} A_N^L\right]\leq \liminf_{N\rightarrow\infty}\mathds{P}\left[A_N^L\right].
		\end{align*}
		From Lemma \ref{231019_lem1}, \eqref{240806_01} and Markov's inequality it follows that $\mathds{P}\left[A_N^L\right]\leq \frac{K_1}{L}$ for all $N\in\mathbb{N}$, hence
		\[\mathds{P}\left[\bigcap_{L\geq 1}\bigcup_{N\in\mathbb{N}}\bigcap_{n\geq N}A_N^L\right]\leq \frac{K_1}{\widetilde L}\]
		for all $\widetilde L\geq 1$ and therefore the assertion follows. Consequently,
		\[\mathds{P}\left[\bigcup_{L\geq 1}\bigcap_{N\in\mathbb{N}}\bigcup_{n\geq N} \left\{\omega\in\Omega  :  \sup_{t\in[0,T]} \Vert u_n(\omega,t,\cdot)\Vert_2^2\leq L\right\}\right]=1.\]
		Hence there exists $N\in\mathcal{F}$ with $\mathds{P}(N)=0$ such that for any $\omega\in N^C$ and any $t \in [0,T]$, there exists $L(\omega)\geq 1$ and a subsequence $(n(\omega))_{n\in\mathbb{N}}$ such that
		\begin{align}\label{240807_01}
			\Vert u_{n(\omega)}(\omega,t)\Vert_2^2\leq \sup_{s\in[0,T]} \Vert u_{n(\omega)}(\omega,s)\Vert_2^2\leq L(\omega).
		\end{align}
		Choosing $\widehat{\Omega}:=\widetilde{\Omega}\cap N^C$, for any fixed $(\omega,t)\in \widehat{\Omega}\times [0,T]$, 
		from the fact that $u_n(\omega,t) \to \overline{u}(\omega,t)$ a.e. in $D$, one concludes that  $u_{n(\omega)}(\omega,t) \rightharpoonup \overline{u}(\omega,t)$ in $L^2(D)$.
		Then, remembering that the a.e. convergence of $u_n(\omega,t)$ to $\overline{u}(\omega,t)$ is non-decreasing, Lebesgue's theorem yields $u_{n}(\omega,t) \to \overline{u}(\omega,t)$ in $L^2(D)$.
		Then, the result holds by relabeling $\widehat{\Omega}$ by $\widetilde{\Omega}$ and since we consider that $\overline{u}$ is the representative we have chosen for $u$.
	\end{proof}
	\begin{remark}
		In the sequel, we will not distinguish between $u$ and its quasi everywhere defined representative $\overline{u}$ if it is not essential for the argumentation.
	\end{remark}
	\subsection{Convergence result}\label{Subsection-convergence}
	From the estimates of Lemma \ref{231019_lem1} we get also weak convergence and regularity results:
	
	\begin{lemma}\label{231019_lem2}
Let $u$ be defined according to Lemmas \ref{240806_lem1} and \ref{110724_rem1}. Then, 
		\[u\in L^2_w(\Omega;L^{\infty}(0,T;L^2(D)))\cap L^p(\Omega;L^p(0,T;W^{1,p}_0(D)))\] 
		and there exists $\Psi\in L^{p'}(\Omega;L^{p'}(0,T;L^{p'}(D)))$ predictable such that for a not relabeled subsequence of $(u_n)_n$, as $n\rightarrow\infty$,
		\begin{align}\label{220724_04}
			u_n \stackrel{\ast}{\rightharpoonup}  u \ \text{weak-}\ast \text{ in} \ L_w^2(\Omega;L^{\infty}(0,T;L^2(D))),
		\end{align}
		\begin{align}\label{231019_02}
			u_n\rightharpoonup u \ \text{weakly in}  \ L^p(\Omega;L^p(0,T;W^{1,p}_0(D))), 
		\end{align}
		\begin{align}\label{231019_04}
			a(x,u_n,\nabla u_n)\rightharpoonup \Psi \ \text{weakly in} \ L^{p'}(\Omega;L^{p'}(0,T;L^{p'}(D))),
		\end{align}
		\begin{align}\label{231019_07}
			u_n^-\rightarrow 0 \ \text{in} \ L^2(\Omega;L^2(0,T;L^2(D))),
		\end{align}
		and a.e. in $Q_T$, and 
		\begin{align}\label{231019_05}
			u_n\rightarrow u \ \text{in} \ L^2(\Omega;L^2(0,T;L^2(D))).
		\end{align}In particular, $u \geq 0$ $d\mathds{P}\otimes dt\otimes dx$-a.e. in $\Omega\times (0,T)\times D$.
	\end{lemma}
	\begin{proof}
		The convergence results \eqref{231019_02}-\eqref{231019_07} are direct consequences of Lemma \ref{231019_lem1}. The convergence result \eqref{220724_04} is a consequence of the boundedness of $u_n$ in $L^2(\Omega; \mathcal{C}([0,T]; L^2(D)))$ and \cite[Theorem 8.20.3]{E95}. 
	\end{proof}
	\begin{lemma}\label{Lemma 220724_01}
		We have
		\begin{align*}
			\erw \sup\limits_{t \in [0,T]} \Vert u(t)\Vert_{L^2(D)}^2 < \infty.
		\end{align*}
		Especially, $\mathds{P}$-a.s. (in a set still denoted $\widetilde\Omega$)  we have
		\begin{align*}
			\sup\limits_{t \in [0,T]} \Vert u(t)\Vert_{L^2(D)}^2 < \infty.
		\end{align*}
	\end{lemma}
\begin{proof}
			From Lemma \ref{110724_rem1}, for all $(\omega,t)\in\widetilde{\Omega}\times [0,T]$, 
			\begin{align*}
				\|u(\omega,t)\|^2_{L^2(D)}=\lim_{n\rightarrow\infty}\|u_n(\omega,t)\|^2_{L^2(D)} \leq \liminf_{n\rightarrow\infty}\Big[\sup_{s\in [0,T]}\|u_n(\omega,s)\|^2_{L^2(D)}\Big],  
			\end{align*}
			so that, \eqref{240806_01} and Fatou's lemma yield
			\begin{align*}
				\erww{\sup_{t\in [0,T]}\|u(\omega,t)\|^2_{L^2(D)}} \leq \liminf_{n\rightarrow\infty}\erww{\sup_{s\in [0,T]}\|u_n(\omega,s)\|^2_{L^2(D)}} \leq K_1.  
			\end{align*}
			Then, being integrable, the random variable $\sup\limits_{t\in[0,T]}\|u(\omega,t)\|^2_{L^2(D)}$ is a.s. finite. 
		\end{proof}
		\begin{remark}\label{u_at_time-t}
				From Lemma \ref{110724_rem1} it follows that the definition of $u(\omega)$ quasi everywhere yields the existence of $u(\omega,t)$ as a measurable function in $D$ for any $t\in [0,T]$. The monotonicity of the approximating sequence ensures it is in $L^2(D)$. Moreover, in the proof of this lemma, one notices that for any $t\in[0,T]$ and all $\omega$ in a subset of $\Omega$ of full measure,  $u_{n}(\omega,t) \to u(\omega,t)$ in $L^2(D)$. Since $|u_n(\omega,t)| \leq |u_1(\omega,t)|+|u(\omega,t)|$ a.e. in $D$, one gets that \[\|u_n(\omega,t)\|^2_{L^2(D)} \leq 2\left(\|u_1(\omega,t)\|^2_{L^2(D)}+\|u(\omega,t)\|^2_{L^2(D)}\right)\in L^1(\Omega)\] 
				and, for any $t\in [0,T]$, $u_{n}(t) \to u(t)$ in $L^2(\Omega\times D)$. 
			\end{remark}
		
	\subsection{Convergence to the reflection measure}
\begin{lemma}\label{lem-cv-reflected}
	There exists $\eta\in L^{p'}(\Omega;\mathcal{W}')$ such that, up to a not relabeled subsequence, 
	\[nu_n^-\rightharpoonup \eta\]
	weakly in $L^{p'}(\Omega;\mathcal{W}')$. 
\end{lemma}
\begin{proof}	
	From \eqref{2},  after integrating by parts it follows that, for any $\xi \in \mathcal{C}_c^{\infty}([0,T]\times D)$, we have
	\begin{align}\label{231019_08}
		\begin{aligned}
			&\int_0^T\int_D nu_n^{-}\xi\,dx\,dt  
			\\ =  &
			\int_D \left(u_n(T) - \int_0^T \Phi \, dW \right)\xi(T) - u_0 \xi(0) \, dx 
			\\&
			- \int_0^T\int_D \left(u_n-\int_0^{\cdot} \Phi\,dW\right)\partial_t \xi\,dx\,dt + \int_0^T\int_D f(u_n)\xi  \,dx\,dt
			\\&
			+\int_0^T\int_D a(x, u_n,\nabla u_n) \cdot\nabla \xi\,dx\,dt .
		\end{aligned}
	\end{align}
	Since $\mathcal{C}_c^{\infty}([0,T]\times D)$ is dense in $\mathcal{W}$ (see Section \ref{A}: Appendix), \eqref{231019_08} holds also true for any $\xi\in \mathcal{W}$ and 
	\[\langle nu_n^-,\xi\rangle_{\mathcal{W'},\mathcal{W}}=\int_0^T\int_D nu_n^{-}\xi\,dx\,dt.\]
	Using H\"older inequality, from \eqref{231019_08} it follows that
	\begin{align}\label{240809_01}
		\begin{aligned}
			&\langle nu_n^-,\xi\rangle_{\mathcal{W'},\mathcal{W}}\\
			&\leq \left(\|u_n(T)\|_{2} + \left\|\int_0^T \Phi \, dW\right\|_2 \right)\|\xi(T)\|_{2} + \|u_0\|_2 \|\xi(0)\|_2 \\
			&+  \left(\|u_n\|_{L^p(0,T;W_0^{1,p}(D))}+\left\|\int_0^{\cdot} \Phi\,dW\right\|_{L^p(0,T;W_0^{1,p}(D))}\right)\|\partial_t \xi\|_{L^{p'}(0,T;W^{-1,p'}(D))}\\
			& + L_f\|u_n\|_{L^2(Q_T)}\|\xi\|_{L^2(Q_T)} +\|a(x, u_n,\nabla u_n)\|_{L^{p'}(0,T;L^{p'}(D)^d)} \| \xi\|_{L^p(0,T;W^{1,p}_0(D))}\\
			&\leq  C_1\left(\|u_n\|_{\mathcal{C}([0,T];L^2(D))} + \left\|\int_0^\cdot \Phi \, dW\right\|_{\mathcal{C}([0,T];L^2(D))}  + \|u_0\|_2\right) \|\xi\|_{\mathcal{C}([0,T];L^2(D))} \\
			&
			+ \left(\|u_n\|_{L^p(0,T;W_0^{1,p}(D))}+\left\|\int_0^{\cdot} \Phi\,dW\right\|_{L^p(0,T;W_0^{1,p}(D))}\right)\|\partial_t \xi\|_{L^{p'}(0,T;W^{-1,p'}(D))} \\
			& + L_f\|u_n\|_{L^2(Q_T)}\|\xi\|_{L^2(Q_T)}  +\|a(x, u_n,\nabla u_n)\|_{L^{p'}(0,T;L^{p'}(D)^d)}
			\|\xi\|_{L^p(0,T;W^{1,p}_0(D))}
		\end{aligned}
	\end{align}
	for a constant $C_1\geq 0$. Since $\mathcal{W}$ is continuously embedded in $\mathcal{C}([0,T];L^2(D))$ (see Section \ref{A}: Appendix), from \eqref{240809_01} it follows that
	\begin{align}\label{240809_02}
		\begin{aligned}
			&\|nu_n^{-}\|_{\mathcal{W}^\prime}\\  
			&\leq C_2\left( \|u_n\|_{\mathcal{C}([0,T];L^2(D))} + \left\|\int_0^\cdot \Phi \, dW\right\|_{\mathcal{C}([0,T];L^2(D))}  + \|u_0\|_2\right)
			+ \|u_n\|_{L^p(0,T;W_0^{1,p}(D))}\\
			&+\left\|\int_0^{\cdot} \Phi\,dW\right\|_{L^p(0,T;W_0^{1,p}(D))} +\|a(x, u_n,\nabla u_n)\|_{L^{p'}(0,T;L^{p'}(D)^d)}.
		\end{aligned}
	\end{align}
	Now, the last two terms in \eqref{240809_02} deserve our attention. Using $(A2)$ and Poincar\'e inequality, we find a constant $C_3\geq 0$ such that
	\begin{align}\label{240809_04}
		\Vert a(x, u_n,\nabla u_n)\Vert_{L^{p'}(0,T;L^{p'}(D)^d)}\leq C_3 (\Vert u_n\Vert^p_{L^p(0,T;W_0^{1,p}(D))}+\Vert g\Vert^{p'}_{L^{p'}(D)})^{1/p'}
	\end{align}
	since $p\geq 2$, there exists constants $C_5\geq 0$ and $C_6\geq 0$ such that
	\begin{align}\label{240809_03}
		\begin{aligned}
			&\|nu_n^{-}\|^{p'}_{\mathcal{W}^\prime}\\
			&\leq C_5 \left(\|u_n\|^{p'}_{\mathcal{C}([0,T];L^2(D))} + \left\|\int_0^\cdot \Phi \, dW\right\|^{p'}_{\mathcal{C}([0,T];L^2(D))}  + \|u_0\|^{p'}_2+\|u_n\|^{p'}_{L^p(0,T;W_0^{1,p}(D))}\right)\\
			&+C_5\left(\left\|\int_0^{\cdot} \Phi\,dW\right\|^{p'}_{L^p(0,T;W_0^{1,p}(D))}+\|u_n\|^{p}_{L^p(0,T;W_0^{1,p}(D))}
			+\Vert g\Vert_{p'}^{p'}\right)\\
			&\leq C_6\left(\|u_n\|^{2}_{\mathcal{C}([0,T];L^2(D))} + \left\|\int_0^\cdot \Phi \, dW\right\|^2_{\mathcal{C}([0,T];L^2(D))}  + \|u_0\|^{2}_2+\|u_n\|^{p}_{L^p(0,T;W_0^{1,p}(D))}\right)\\
			&+ C_6\left(\left\|\int_0^{\cdot} \Phi\,dW\right\|^{p'}_{L^p(0,T;W_0^{1,p}(D))}+\Vert g\Vert_{p'}^{p'}+1\right).
		\end{aligned}
	\end{align}

	Taking expectation in \eqref{240809_03}, using Lemma \ref{231019_lem1} and the assumptions on the noise term, it follows that $(nu_n^-)_{n\in\mathbb{N}}$ is bounded in $L^{p'}(\Omega,\mathcal{W}')$. Therefore we find the weak convergence (up to a subsequence) of  $(nu_n^-)_{n\in\mathbb{N}}$ towards an element $\eta$ weakly in $L^{p'}(\Omega,\mathcal{W}')$. Since $nu_n^{-}\geq 0$ $d\mathds{P}\otimes dt\otimes dx$-a.e. in $\Omega\times Q_T$, for any $\xi\in \mathcal{W}^+:=\{\xi\in\mathcal{W} : \xi\geq 0\}$ and all $A\in\mathcal{F}$ we have
	\[0\leq \lim_{n\rightarrow\infty} \int_A\int_{Q_T}nu_n^-\xi \, d(t,x)\,d\mathds{P}=\int_A\langle \eta,\xi\rangle_{\mathcal{W'},\mathcal{W}}\,d\mathds{P}\] 
	hence $\langle \eta,\xi\rangle_{\mathcal{W'},\mathcal{W}}\geq 0$ $\mathds{P}$-a.s in $\Omega$, where the exceptional set may depend on $\xi$. Since $\mathcal{W}$ is separable, there exists a countable dense subset $C\subset \mathcal{W}$, and $N\in\mathcal{F}$ with $\mathds{P}(N)=0$ such that for all $\omega\in\dot{\Omega}:=\Omega\setminus N$, for all $\xi\in C\cap \mathcal{W}^+$ we have 
	\[\langle \eta(\omega),\xi\rangle_{\mathcal{W'},\mathcal{W}}\geq 0.\] 
	By continuity of $\xi\mapsto \langle \eta(\omega),\xi\rangle_{\mathcal{W'},\mathcal{W}}$, this inequality may be extended to all $\xi\in\mathcal{W}^+$, and in particular for all $\xi\in \mathcal{D}(Q_T)$ such that $\xi\geq 0$. Now, from \cite[Theorem 2.1.7]{H03}), it follows that $\eta(\omega)$ is a non-negative Radon measure on $Q_T$ for all $\omega\in\dot{\Omega}$.
	
\end{proof}
	
		For $A\in\mathcal{F}$, multiplying \eqref{231019_08} with $\mathds{1}_A$ and integrating over $\Omega$, we arrive, thanks to the convergence results of Lemma \ref{231019_lem2} and Lemma \ref{110724_rem1}, at
		\begin{align}\label{220724_03}
			\begin{aligned}
				&\int_A\int_D \left(u(T) - \int_0^T \Phi \, dW \right)\xi(T) - u_0 \xi(0) \, dx \, d\mathds{P} \\
				&-\int_A\int_0^T\int_D \left(u-\int_0^{\cdot} \Phi\,dW\right)\partial_t\xi\,dx\,dt\,d\mathds{P} +\int_A\int_0^T\int_D f(u)\xi \, dx\,dt\,d\mathds{P} \\
				&+\int_A\int_0^T\int_D \Psi \cdot\nabla \xi\,dx\,dt\,d\mathds{P} \\
				&=\int_A\langle\eta ,\xi\rangle_{\mathcal{W}', \mathcal{W}}\,d\mathds{P},
			\end{aligned}
		\end{align}
		and therefore, thanks again to an argument of separability as above,  $\mathds{P}$-a.s. in $\Omega$
		\begin{align}\label{5}
			\partial_t (u - \int_0^{\cdot} \Phi \, dW) - \diver \Psi + f(u) = \eta
		\end{align}
		in $\mathcal{W}'$.
		Let us remark that $\mathcal{D}(Q_T) \subset \mathcal{W}$, hence equation \eqref{5} is also satisfied in $\mathcal{D}'(Q_T)$ and $\eta$ is a random element in $\mathcal{D}'(Q_T)$.
				
		\begin{lemma}\label{Lemma 2}
			We have
			\begin{align*}
				\lim\limits_{n \to \infty} \mathbb{E} \sup\limits_{t \in [0,T]} \bigg| \int_0^t (u_n - u, \Phi \, dW)_2 \bigg| =0.
			\end{align*}
Especially, there exist a not relabeled subsequence,  $0\leq \frak{L} \in L^1(\Omega)$  and 
a set of full measure still denoted $\widetilde\Omega$ such that, for any $\omega \in \widetilde{\Omega}$, 
			\begin{align}\label{limit}
			\lim\limits_{n \to \infty} \sup_{t\in[0,T]} \left |\int_0^t (u_n-u, \Phi \, dW)_2 \right|=0
			\end{align}
				and, for this subsequence,
			\begin{align}\label{majoration}
				\sup_{t \in [0,T]}  \big| \int_0^t (u_n, \Phi \, dW)_2 \big| \leq \frak{L}.
			\end{align}
			for all  $n \in \N$.
		\end{lemma}
		\begin{proof}
			According to Burkholder's inequality, we have
			\begin{align*}
				\mathbb{E} \Big[\sup\limits_{t \in [0,T]} \bigg| \int_0^t (u_n - u, \Phi \, dW)_2 \bigg|\Big] \leq 3\mathbb{E} \Big[\bigg( \int_0^T \Vert u_n - u \Vert_2^2 \cdot \Vert \Phi \Vert_{HS(L^2(D))}^2 \, dt \bigg)^{\frac{1}{2}}\Big].
			\end{align*}
			By Lemma \ref{110724_rem1}, 
			\begin{align*}
			\lim_{n\to\infty}	\Vert u_n(t) - u(t) \Vert_2^2 \cdot \Vert \Phi(t) \Vert_{HS(L^2(D))}^2 =0
			\end{align*}
			and, there exists a constant $C\geq 0$ such that, $\mathds{P}$-a.s.,   
			\begin{align*}
				&\Vert u_n(t) - u(t) \Vert_2^2
				\leq 2 \left(\Vert u_n(t)\|^2_2 + \|u(t)\|_2^2\right)
				\leq C \left( \Vert u_1(t)\|^2_2 + \|u(t)\|_2^2\right) \\
				&\leq C \left(\sup_{t\in [0,T]}\Vert u_1(t)\|^2_2 + \sup_{t\in [0,T]}\|u(t)\|_2^2\right) <+\infty.
			\end{align*}
			Then, Lebesgue's theorem yields, $\mathds{P}$-a.s., 
			\[\lim_{n\to\infty} \int_0^T \Vert u_n - u \Vert_2^2 \cdot \Vert \Phi \Vert_{HS(L^2(D))}^2 \, dt =0.\] 
			 Moreover, H\"{o}lder's and Young's inequality yield
			\begin{align*}
				& \left(\int_0^T \Vert u_n - u \Vert_2^2 \cdot \Vert \Phi \Vert_{HS(L^2(D))}^2 \, dt\right)^{1/2} \\
				&\leq  \left(C\left(\sup_{t\in [0,T]}\Vert u_1(t)\|_2^2 + \sup_{t\in [0,T]}\|u(t)\|_2^2\right)\right)^{1/2}
				\left(\int_0^T  \Vert \Phi \Vert_{HS(L^2(D))}^2 \, dt\right)^{1/2}
				\\ &\leq 
				C \left( \sup_{t\in [0,T]}\Vert u_1(t)\|^2_2 + \sup_{t\in [0,T]}\|u(t)\|^2_2\right)   + 
				\int_0^T  \Vert \Phi \Vert_{HS(L^2(D))}^2 \, dt \in L^1(\Omega).
			\end{align*}
			Again, Lebesgue's theorem yields 
			\[\lim_{n\rightarrow\infty}\erww{\sup\limits_{t \in [0,T]} \bigg| \dint_0^t (u_n - u, \Phi \, dW)_2 \bigg|}=0.\]
		\end{proof}
		
		\begin{remark}\label{Remark 120724_01}
			Considering  $\frak{L}$ given in \eqref{majoration}, denote by 
			\begin{align*}
				\frak{M}=\frac{1}{2} \int_D u_0^2 \, dx + \frak{L} + \frac{1}{2}\int_0^T \Vert \Phi(s)\Vert^2_{\operatorname{HS}} \,dt \in L^1(\Omega).
			\end{align*}
			Lemma \ref{Lemma 2}, Assumptions $(A1)$-$(A3)$ and equality \eqref{4} yield that, $\mathds{P}$-a.s., the boundedness results in Lemma \ref{231019_lem1} are still valid without taking expectation: 
			there exists $$\frak{K}=\int_D\vert\kappa(x)\vert	dx+ \frak{M}\in L^1(\Omega),$$ not depending on $n \in \N$, such that, for any $\omega$ in an updated $\widetilde{\Omega}$ of full measure in $\Omega$, 
			\begin{align*}
				\sup_{t\in[0,T]}\Vert u_n(t)\Vert_2^2  
				+ \Vert \nabla u_n\Vert_{L^p(Q_T)}^p
				+ \Vert a(\cdot,u_n,\nabla u_n)\Vert_{L^{p'}(Q_T)}^{p'}
				+ \Vert \sqrt{n} u_n^{-}\Vert_{L^{2}(Q_T)}^{2}
				\leq \frak{K}<+\infty.
			\end{align*}
		\end{remark}

\section{$L^1$-estimates and convergence of the gradients}\label{section-L1-estimate}\label{s4}

In the following, we fix the function $\varphi:\overline{D} \to \mathbb{R}$, $\varphi(x):= \min (1, \operatorname{dist}(x, \partial D))$. Then we have $\varphi \in \mathcal{C}_0(\overline{D})\cap W_0^{1,p}(D)$, $0 < \varphi\leq 1$ in $D$ and $|\nabla \varphi| \leq 1$ in $D$. This function can easily be extended to an element of $\mathcal{W}$ and to a measurable function $\tilde{\varphi}: \Omega \times [0,T] \times \overline{D} \to \R$ defined by $\tilde{\varphi}(\omega, t, x):= \varphi(x)$ which will still be denoted by $\varphi$ in the following. Note that for any $v \in \mathcal{W}$, $v\varphi \in \mathcal{W}$ also, since $v\varphi \in L^p(Q_T)$ and
\begin{align*}
&\nabla (v\varphi) = v \nabla \varphi + \varphi \nabla v \in L^p(Q_T)
\\
&\partial_t (v \varphi) = \varphi \partial_t v \in L^{p'}(0,T;W^{-1,p'}(D))\text{ since, }\forall \psi \in W^{1,p}_0(D),\ \langle \partial_t (v \varphi) , \psi \rangle = \langle \partial_t v , \varphi \psi \rangle.
\end{align*}
\begin{lemma}\label{Lemma 230224_01}
The sequence $(n u_n^- \varphi)_n$ is bounded in $L^1(\Omega \times Q_T)$ and there exists $\frak{K_2}\in L^1(\Omega)$ such that $\mathds{P}$-a.s. $\|n u_n^- \varphi\|_{L^1(Q_T)} \leq \frak{K_2}$.
\end{lemma}
\begin{proof}
Testing \eqref{2} with $\varphi$ and setting $t=T$ yields $\mathds{P}$-a.s.
\begin{align*}
n \int_{Q_T} \underbrace{u_n^- \varphi}_{\geq 0} \, dx \, dt = &\int_D u_n(T) \varphi \, dx - \int _D u_0 \varphi \, dx +\int_{Q_T} a(x, u_n, \nabla u_n) \nabla \varphi \, dx \,  dt \\
&- \int_{Q_T} f(u_n) \varphi \, dx \, dt - \int_D \bigg( \int_0^T \Phi \, dW\bigg) \varphi \, dx.
\end{align*}
Now, the boundedness of the right hand side of the equality $\mathds{P}$-a.s., as developed in Remark \ref{Remark 120724_01},  yields the second assertion. The first one follows by taking the expectation and Lemma \ref{231019_lem1}.
\end{proof}
\begin{lemma}\label{Lemma 190724_01}
We have
\begin{align*}
\nabla u_n \to ~\nabla u ~\text{in measure on}~ \Omega \times Q_T.
\end{align*}
Especially, for a subsequence $\nabla u_n \to \nabla u$ a.e. in $\Omega \times Q_T$.
\end{lemma}
\begin{proof}
Set $k \in (0,1)$. We have $\mathds{P}$-a.s. in $\Omega$
\begin{align*}
&\partial_t ( u_n - u_m) - \bigg(\operatorname{div} [a(\cdot,u_n, \nabla u_n) - a(\cdot,u_m, \nabla u_m) ]\bigg) + f(u_n) - f(u_m) \\
&= n u_n^- - m u_m^-.
\end{align*}
Testing with $T_k(u_n - u_m) \varphi$, setting $\tilde{T}_k(r):= \int_0^r T_k(r') \, dr'$ and taking expectation yield:
\begin{align*}
&\erw \int_D \tilde{T}_k(u_n(t) - u_m(t)) \varphi \, dx \\
&+ \erw \int_0^t \int_D \bigg(a(x ,u_n, \nabla u_n) - a(x , u_m, \nabla u_m)\bigg) \nabla (T_k(u_n - u_m) \varphi) \, dx \, ds\\
&+ \erw \int_0^t \int_D \bigg(f(u_n) - f(u_m)\bigg) T_k(u_n - u_m) \varphi \, dx \, ds\\
&= \erw \int_0^t \int_D \bigg(n u_n^- - m u_m^-\bigg)T_k(u_n - u_m) \varphi \, dx \, ds\\
&\Leftrightarrow I_1+I_2+I_3=I_4.
\end{align*}
We get
\begin{align*}
I_1 \geq& 0, \\
I_2 =& \erw\int_0^t \int_D \varphi \mathds{1}_{\{|u_n - u_m| \leq k\}} \bigg(a(x , u_n, \nabla u_n) - a(x , u_m,\nabla u_m)\bigg)  \nabla (u_n - u_m) \, dx \, ds\\
&+ \erw \int_0^t \int_D T_k(u_n - u_m) \bigg(a(x , u_n, \nabla u_n) - a(x ,  u_m, \nabla u_m)\bigg)  \nabla \varphi \, dx \, ds = I_{2,1} + I_{2,2},
\end{align*}
\begin{align*}
|I_{2,2}| \leq& k \erw \int_{Q_T} |a(x , u_n,\nabla u_n)| + |a(x , u_m, \nabla u_m)| \, d(t,x) 
\\ \leq& Ck \erw \Big(\int_{Q_T} C_2|\nabla u_n|^{p-1} + C_2|\nabla u_m|^{p-1} + C_3| u_n|^{p-1} + C_3| u_m|^{p-1} + g(x)
\, d(t,x) \Big)
\\ \leq& Ck (1+ \erw \Vert u_n \Vert_{W^{1,p}_0(D)}^p + \erw \Vert u_m \Vert_{W^{1,p}_0(D)}^p) \leq \tilde C_1 k.
\end{align*}
Note that
\begin{align*}
I_{2,1}=&\erw\int_0^t \int_D \varphi \mathds{1}_{\{|u_n - u_m| \leq k\}} \bigg(a(x , u_n, \nabla u_n) - a(x , u_m, \nabla u_m)\bigg)  \nabla (u_n - u_m) \, dx \, ds
\\ =&
\erw\int_0^t \int_D \varphi \mathds{1}_{\{|u_n - u_m| \leq k\}} \bigg(a(x , u_n,\nabla u_n) - a(x ,u_m, \nabla u_n)\bigg)  \nabla (u_n - u_m) \, dx \, ds
\\ &+
\erw\int_0^t \int_D \varphi \mathds{1}_{\{|u_n - u_m| \leq k\}} \underbrace{\bigg(a(x , u_m, \nabla u_n) - a(x , u_m,\nabla u_m)\bigg)  \nabla (u_n - u_m)}_{\geq 0} \, dx \, ds,
\end{align*}
where, for the first term, 
\begin{align*}
&\Big|\erw\int_0^t \int_D \varphi \mathds{1}_{\{|u_n - u_m| \leq k\}} \bigg(a(x, u_n,\nabla u_n) - a(x,u_m, \nabla u_n)\bigg)  \nabla (u_n - u_m) \, dx \, ds\Big|
\\ \leq &
\erw\int_0^t \int_D \varphi \mathds{1}_{\{|u_n - u_m| \leq k\}}  |u_n - u_m|\Big(C_4 |\nabla u_n|^{p-1}+h(x)\Big)|\nabla (u_n - u_m)| \, dx \, ds
\\ \leq &
Ck\erw\int_0^t \int_D  |\nabla u_n|^{p}+h^{p'}(x)+ |\nabla (u_n - u_m)|^p \, dx \, ds \leq \tilde C_{1,1}k,
\end{align*}
\begin{align*}
|I_3| &= \bigg| \erw \int_0^T \int_D \bigg(f(u_n) - f(u_m)\bigg) T_k(u_n - u_m) \varphi \, dx \, dt\bigg| \leq L_f k \erw \Vert u_n - u_m \Vert_1 \leq \tilde C_3 k,
\end{align*}
\begin{align*}
I_4 &\leq \erw \int_0^t \int_D (n- m) u_m^- T_k(u_n - u_m) \varphi \, dx \, ds = \erw \int_0^t \int_D (m-n) u_m^- T_k(u_m - u_n) \varphi \, dx \, ds\\
&\leq k (m-n) \erw \int_0^t \int_D u_m^- \varphi \, dx \, ds \leq k m \erw \int_0^t \int_D u_m^- \varphi \, dx \, ds \leq \tilde C_4 k
\end{align*}
where $m \geq n$ and for some constants $\tilde C_1,\tilde C_{1,1},\tilde C_2,\tilde C_3,\tilde C_4 >0$. Therefore, from the last term of $I_{2,1}$, a constant $\tilde C>0$ exists such that, for any $k \in (0,1)$,  
\begin{align*}
0 \leq \erw\int_0^t \int_D \varphi \bigg(a(x, u_m, \nabla u_n) - a(x, u_m,\nabla u_m)\bigg)  \nabla T_k(u_n - u_m) \, dx \, ds \leq \tilde C k
\end{align*}
Let us set $\ell >0$, and note that 
\begin{align*}
&\{ |\nabla u_n - \nabla u_m| \geq \ell \} \subset \{ |u_n| + |\nabla u_n| + |u_m| + |\nabla u_m| >M \} \cup  \{ | u_n -  u_m| \geq k \}
\\ &\cup \{| u_n -  u_m| < k,  |u_n| + |u_m| + |\nabla u_n| + |\nabla u_m| \leq M, |\nabla u_n - \nabla u_m| \geq \ell \}
\\ &= A^M_1 \cup A^k_2 \cup A^{k,M,\ell}_3.
\end{align*}
By Markov's inequality, a positive constant $C$ exists such that 
\begin{align*}
\Big|A^M_1\Big| = \Big|  \{ |u_n| + |\nabla u_n| + |u_m| + |\nabla u_m| >M \}\Big| \leq \frac{C}{M^p}
\end{align*}
Let $\theta >0$ and fix $M$ such that $\frac{C}{M^p} \leq \frac{\theta}3$. Denote by $K_{M,\ell}$ the compact set 
\begin{align*}
K_{M,\ell}=\{(\lambda_1,\lambda_2,\xi_1,\xi_2)\in\R^{2+2d}, |\lambda_1|+|\lambda_2|+|\xi_1|+|\xi_2|\leq M, |\xi_1-\xi_2|\geq \ell \}
\end{align*}
 and note that for a fixed $x\in D$, $(\lambda_1,\lambda_2,\xi_1,\xi_2) \mapsto (a(x,\lambda_1,\xi_1)-a(x,\lambda_2,\xi_2))\cdot(\xi_1-\xi_2)$ is a continuous function satisfying $(a(x,\lambda_2,\xi_1)-a(x,\lambda_2,\xi_2))\cdot(\xi_1-\xi_2)> 0$ in $K_{M,\ell}$. Denote by $\gamma_{M,\ell}(x)>0$ the infimum of this function in $K_{M,\ell}$ and note that $\gamma_{M,\ell}$ is measurable as a function of variable $x$ in $D$. We recall that  $A^{k,M,\ell}_3=A^{M,\ell} \cap \{| u_n -  u_m| < k\}$ where 
 \begin{align*}
A^{M,\ell}&=\{|u_n| + |u_m| + |\nabla u_n| + |\nabla u_m| \leq M, |\nabla u_n - \nabla u_m| \geq \ell \}\\
&= \{(u_n, u_m,\nabla u_n,\nabla u_m)\in K_{M,\ell}\}
\end{align*}
so that one has
 \begin{align*}
&\erw\int_{A^{k,M,\ell}_3} \gamma_{M,\ell}(x) \varphi \, d(t,x) \\
&\leq \erw\int_{Q_T} \mathds{1}_{A^{k,M,\ell}_3}\varphi \bigg(a(x, u_m, \nabla u_n) - a(x, u_m,\nabla u_m)\bigg)  \nabla T_k(u_n - u_m) \, d(t,x)
\\ &\leq \erw\int_{Q_T} \varphi \bigg(a(x, u_m, \nabla u_n) - a(x, u_m,\nabla u_m)\bigg)  \nabla T_k(u_n - u_m) \, d(t,x) \\
&\leq \tilde C k.
\end{align*}
Let us recall \cite[Lemma 2]{Boccardo}: 
\begin{lemma}
Let $(\mathrm{X}, \mathrm{T}, \mathrm{m})$ be a measurable space such that $\mathrm{m}(\mathrm{X})<\infty$. Let $\beta: \mathrm{X} \rightarrow[0,+\infty]$ be a measurable function such that $\mathrm{m}(\{\mathrm{x} \in \mathrm{X}, \beta(\mathrm{x})=0\})=0$. Then for any $\varepsilon>0$, there exists $\delta>0$ such that for all $A \in \mathrm{T}$ we have
\begin{align*}
\int_{\mathrm{A}} \beta\, \mathrm{dm} \leq \delta  \Rightarrow m(A) \leq \varepsilon.
\end{align*}
\end{lemma}
For $X=\Omega\times Q_T$ with measure $d\mathds{P}\otimes dt\otimes dx$ and $\beta=\varphi\gamma_{M,\ell}$, then, for $\varepsilon = \frac\theta3$, there exists $\delta_{M,\ell,\varphi}>0$ such that for all $0 \leq k \leq \frac{\delta_{M,\ell,\varphi}}{\tilde C}$ we have $\Big|A^{k,M,\ell}_3\Big| \leq \frac\theta3$.
Now, choose $0 \leq k\leq \frac{\delta_{M,\ell,\varphi}}{\tilde C}$. Since $(u_n)_n$ converges a.e. in $\Omega\times Q_T$, it converges in measure so that there exists $n_0 \in \N$ such that 
\begin{align*}
\forall n, m \in \N,\quad n,m \geq n_0 \implies \Big| \{ | u_n -  u_m| \geq k \Big| \leq \frac{\theta}3.
\end{align*}
Hence we have $|\{|\nabla u_n - \nabla u_m| \geq l \} | \leq \theta$ and one concludes that $(\nabla u_n)_n$ is a Cauchy-sequence in measure, thus converges in measure in  $\Omega\times Q_T$ and a.e. in  $\Omega\times Q_T$ for a subsequence. 
\end{proof}
Lemma \ref{Lemma 190724_01} yields that for a subsequence $\nabla u_n \to \nabla u$ a.e. in $\Omega \times Q_T$ and therefore $a(\cdot, u_n, \nabla u_n) \to a(\cdot, u, \nabla u)$ a.e. in $\Omega \times Q_T$.  
Moreover, for a set of full measure $\widetilde{\Omega}$, for any $\omega \in \widetilde{\Omega}$, $a(\cdot, u_n(\omega), \nabla u_n(\omega)) \to a(\cdot, u(\omega), \nabla u(\omega))$ a.e. in $Q_T$.
Now, the weak convergence \eqref{231019_04} yields $\Psi = a(\cdot, u, \nabla u)$ in $L^{p'}(\Omega \times Q_T)$; and by Vitali's theorem,  $a(\cdot, u_n, \nabla u_n) \to a(\cdot, u, \nabla u)$ in $L^r(\Omega \times Q_T)$ for any $r \in [1,p')$. On the other hand, for another subsequence and a new set of full measure still denoted $\widetilde{\Omega}$, $ a(\cdot, u_n(\omega), \nabla u_n(\omega)) \to a(\cdot, u(\omega), \nabla u(\omega))$ in $L^r(Q_T)$ for any $\omega \in \widetilde{\Omega}$. Thanks to Remark \ref{Remark 120724_01}, for any $\omega \in \widetilde{\Omega}$, $\|a(\cdot, u_n(\omega), \nabla u_n(\omega))\|_{L^{p'}(Q_T)}^{p'} \leq \frak{L} <+\infty$ and  there exists $\chi(\omega) \in L^{p'}(Q_T)$ and a subsequence $n_k(\omega)$  such that $ a(\cdot, u_{n_k(\omega)}(\omega), \nabla u_{n_k(\omega)}(\omega)) \rightharpoonup \chi(\omega)$ in $L^{p'}(Q_T)$. Especially, this weak convergence holds true in $L^r(Q_T)$ which yields $\chi(\omega)= a(\cdot, u(\omega), \nabla u(\omega))$. Now, the subsequence principle yields that $ a(\cdot, u_n(\omega), \nabla u_n(\omega)) \rightharpoonup a(\cdot, u(\omega), \nabla u(\omega))$ in $L^{p'}(Q_T)$ for all $\omega \in \widetilde{\Omega}$. This argumentation proves the following corollary
\begin{corollary}\label{Corollary 190724_01}
We have $a(\cdot, u , \nabla u) = \Psi$ in $L^{p'}(\Omega \times Q_T)$ and, up to a subsequence, $\mathds{P}$-a.s. in $\Omega$,
\begin{align*}
a(\cdot, u_n, \nabla u_n) \rightharpoonup a(\cdot, u , \nabla u) ~\text{in}~L^{p'}(Q_T).
\end{align*}
\end{corollary}
Moreover, we have the following result:
\begin{lemma}\label{CVWprime}
$\mathds{P}$-a.s. in $\Omega$, we have
\begin{align*}
\eta_n \rightharpoonup \eta ~\text{in}~\mathcal{W}'.
\end{align*}
\end{lemma}
\begin{proof}
Going back to \eqref{231019_08}, since $A \in \mathcal{F}$ can be chosen arbitrary we obtain $\mathds{P}$-a.s. in $\Omega$
\begin{align}\label{220724_01}
\begin{aligned}
    &\int_D \left(u_n(T) - \int_0^T \Phi \, dW \right)\xi(T) - u_0 \xi(0) \, dx \\
    &- \int_0^T\int_D \left(u_n-\int_0^{\cdot} \Phi\,dW\right)\partial_t \xi\,dx\,dt+\int_0^T\int_D f(u_n) \xi \,dx\,dt\\
    &+\int_0^T\int_D a(x,u_n,\nabla u_n) \cdot\nabla \xi\,dx\,dt\\
    &=\int_0^T\int_D nu_n^{-}\xi\,dx\,dt.
\end{aligned}
\end{align}
The convergence results of Lemma \ref{231019_lem2}, Remark \ref{110724_rem1} and Corollary \ref{Corollary 190724_01} yield that we can pass to the limit in \eqref{220724_01} and we obtain $\mathds{P}$-a.s.
\begin{align}\label{220724_02}
\begin{aligned}
    &\int_D \left(u_T - \int_0^T \Phi \, dW \right)\xi(T) - u_0 \xi(0) \, dx \\
    &- \int_0^T\int_D \left(u-\int_0^{\cdot} \Phi\,dW\right)\partial_t \xi\,dx\,dt+\int_0^T\int_D f(u)\xi  \, dx\,dt\\
    &+\int_0^T\int_D a(x,u,\nabla u)\cdot\nabla \xi\,dx\,dt\\
    &=\langle \tilde{\eta}, \xi \rangle_{\mathcal{W}', \mathcal{W}}
\end{aligned}
\end{align}
for some random element $\tilde{\eta} \in \mathcal{W}'$. Since \eqref{220724_03} is valid for arbitrary $A \in \mathcal{F}$, the comparison of \eqref{220724_03} and \eqref{220724_02} yields $\mathds{P}$-a.s. in $\Omega$ and for all $\xi \in \mathcal{W}$
\begin{align*}
\langle \tilde{\eta}, \xi \rangle_{\mathcal{W}', \mathcal{W}}= \langle \eta, \xi \rangle_{\mathcal{W}', \mathcal{W}},
\end{align*}
hence $\tilde{\eta} = \eta$.
\end{proof}
\begin{lemma}\label{241204_lem01}
	$(\eta(\omega))_{\omega\in\Omega}$ is a weak-$\ast$ adapted, random non-negative Radon measure.
\end{lemma}
\begin{proof}
		Consider $t_0 \in (0,T)$, $\xi \in \mathcal{C}^\infty_c(Q_T)$ and $\varphi_k$ a non-increasing $\mathcal{C}^\infty_c([0,T))$ function satisfying $\varphi_k(t)=1$ if $t\leq t_0$ and $\varphi_k(t)=0$ if $t\geq t_0+\frac1k$ (for $k\in \mathbb{N}$ large enough).
		Then, $\xi\varphi_k \in \mathcal{C}^\infty_c(Q_T)$ and, using the convergence result of Lemma \ref{CVWprime}, $\mathds{P}$-a.s. we have 
		\begin{align*}
			&n \dint_{(0,t_0+\frac1k]\times D} \xi\varphi_k u_n^- \ d(t,x) = n \dint_{Q_T} \xi\varphi_k u_n^- \ d(t,x)\\ &\to  \langle \eta, \xi\varphi_k \rangle_{\mathcal{W}', \mathcal{W}} = \dint_{Q_T} \xi\varphi_k \ d \eta = \dint_{(0,t_0+\frac{1}{k}]\times D} \xi\varphi_k \ d \eta .
		\end{align*}
		Therefore, 
		\[\dint_{(0,t_0+\frac{1}{k}]\times D} \xi\varphi_k \ d \eta\] 
		is $\mathcal{F}_{t_0+\frac1k}$ measurable. Now, we fix $k_0\in\mathbb{N}$ large enough. Then, $\mathds{P}$-a.s., one may apply Lebesgue's theorem to pass to the limit for $k \to +\infty$. In this way it follows that
		 \[\dint_{(0,t_0]\times D} \xi \ d \eta\] 
		 is $\mathcal{F}_{t_0+\frac{1}{k_0}}$ measurable for any $k_0\in\mathbb{N}$ large enough, thus $\mathcal{F}_{t_0^+}$ measurable, \textit{i.e.} $\mathcal{F}_{t_0}$ measurable by usual assumptions on the filtration. Finally, if $\xi \in \mathcal{C}_c(Q_T)$, since it can be uniformly approximated by a sequence $(\xi_n)_n \subset \mathcal{C}^\infty_c(Q_T)$, one gets that 
		\[\dint_{(0,t_0]\times D} \xi \ d \eta=\lim_{n\rightarrow\infty}\dint_{(0,t_0]\times D} \xi_n \ d \eta\] 
		and the $\mathcal{F}_{t_0}$ measurability is kept at the limit.
		\end{proof}

In the sequel, let $\varphi\in \mathcal{C}^\infty_c(Q_T)$ be a non-negative weight function for the measure $\eta$.
\begin{lemma}\label{narrowconvergence_241011}
For $0\leq \varphi \in \mathcal{C}^\infty_c(Q_T)$, set $\eta_{n}^\varphi=n u_{n}^- \varphi$ and  $\eta^\varphi:\W \to\R$,  $v\mapsto \langle \eta,v \varphi \rangle$. For a set $\widetilde \Omega_\varphi$ of full measure in $\Omega$, $\eta_{n}^\varphi$ converges weakly to $\eta^\varphi$ in $\mathcal{W}'$ and narrowly in the sense of Radon measures to a non-negative measure still denoted $\eta^\varphi$. 
\end{lemma}
\begin{proof}
Note first that since $v \varphi \in \mathcal{W}$ for any $v \in \mathcal{W}$, Lemma \ref{CVWprime} yields the weak convergence of $\eta_{n}^\varphi$ to $\eta^\varphi$ in $\mathcal{W}'$. With similar arguments as in Lemma \ref{Lemma 230224_01} it follows that $0 \leq n u_n^- \varphi$ is bounded in $L^1(\Omega \times Q_T)$ and $\mathds{P}$-a.s. bounded by $\frak{K_3}(\omega)$, for $\frak{K_3} \in L^1(\Omega)$. As a consequence, $\mathds{P}$-a.s., there exists a subsequence $n_{k(\omega)}$, such that $\eta_{n_{k(\omega)}}^\varphi :=n_{k(\omega)} u_{n_{k(\omega)}}^- \varphi$ converges weakly in the sense of Radon measures to a non-negative measure $\mu_\omega$. In particular, for any $v \in \mathcal{C}_c(Q_T)\cap\mathcal{W}$,  $v \varphi \in \mathcal{W}$ and 
\begin{align}\label{CVdistrib}
\langle \eta  , \varphi v \rangle_{\mathcal{W}^\prime,\mathcal{W}} =\lim_{k(\omega)\to\infty} \langle \eta_{n_{k(\omega)}}^\varphi  , v \rangle_{\mathcal{W}^\prime,\mathcal{W}} = \lim_{k(\omega)\to\infty} \int_{Q_T}  v \eta_{n_{k(\omega)}}^\varphi  d(t,x) = \int_{Q_T}v d\mu_\omega.
\end{align}
From \eqref{CVdistrib} it follows that there exists a full measure set $\widetilde{\Omega}_\varphi$ of $\Omega$ such that
\[\eta^{\varphi}(\omega)=\mu_{\omega} \ \text{in} \ \mathcal{D}'(Q_T)\]
for all $\omega\in\widetilde{\Omega}_\varphi$. Since $\eta^{\varphi}$ is a non-negative distribution, it is a Radon measure and therefore $\eta^{\varphi}(\omega)=\mu_{\omega}$ as Radon measures for all $\omega\in\widetilde{\Omega}_\varphi$. Thus, the whole sequence $(\eta^{\varphi}_n)_n$ converges to $\eta^{\varphi}$ weakly in the sense of Radon measures.
Let $K \subset Q_T$ be the compact support of $\varphi$ and $0 \leq \Theta \in \mathcal{C}_c(Q_T)$ with $\Theta=1$ in $K$. Then, for any $\xi \in \mathcal{W}\cap \mathcal{C}(\overline{Q}_T)$,  the definition of  $\eta_\varphi$ yields
\begin{align}\label{241211_01}
	&\int_{Q_T} 1 \,d\eta^{\varphi}=\int_K 1  \,d\eta^{\varphi}=\int_K \Theta  \,d\eta^{\varphi} =\int_{Q_T}\Theta  \,d\eta^{\varphi}=\langle \eta  , \varphi \Theta \rangle_{\mathcal{W}^\prime,\mathcal{W}}=\int_{Q_T}\varphi\Theta \,d\eta =\int_{Q_T}\varphi \,d\eta
	\end{align}
Therefore, using \eqref{241211_01}, it follows that
\begin{align*}
\lim_{n\to\infty}\int_{Q_T}1 \,d\eta_{n}^\varphi=\lim_{n\to\infty}\langle n u_{n}^- ,\varphi \rangle_{\mathcal{W}^\prime,\mathcal{W}}= \langle \eta ,\varphi \rangle_{\mathcal{W}^\prime,\mathcal{W}} = \int_{Q_T}\varphi \,d\eta=\int_{Q_T}1 \,d\eta^\varphi.
\end{align*}
and the narrow convergence now follows by \cite[Proposition 9.9.3]{M74}.
\end{proof}
\begin{remark}\label{Rem4.10}
$\mathds{P}$-a.s., $\eta^\varphi$ is a bounded non-negative Radon measure and $\eta$ is a non-negative Radon measure such that
\[\dint_{Q_T} \xi \, d\eta^\varphi = \langle \eta^\varphi,\xi \rangle= \langle \eta,\xi\varphi \rangle = \int_{Q_T} \xi \varphi \,d\eta\]
for all $\xi \in \mathcal{C}^1_c(Q_T)$. Therefore, $\eta^\varphi$ is absolutely continuous with respect to $\eta$ with density $\varphi$. So that the above equality of integrals holds for any non-negative (or bounded, respectively) Borel function $\xi$ on $Q_T$. Moreover, from Lemma \ref{lem-cv-reflected} it immediately follows that 
\begin{align*}
	\eta_n^{\varphi} \rightharpoonup \eta^{\varphi}~\text{in}~L^{p'}(\Omega, \mathcal{W}').
\end{align*}
\end{remark}
\section{Properties of $u$ and $\eta$}\label{Section-reflected+solution}\label{s5}
\begin{proposition}\label{cadlag}
As a complement to Lemmas \ref{110724_rem1}, \ref{Lemma 220724_01} and Remark \ref{u_at_time-t}, which state the existence of $u(t)$ in $L^2(\Omega \times D)$ for any $t \in [0,T]$, and,  almost surely in $\Omega$,  the existence of $u(\omega,t) \in L^2(D)$ for any $t \in [0,T]$ for the representative of $u$ we consider, one has that $\mathds{P}$-a.s., the mapping $t \mapsto u(t)$ is weakly c\'{a}dl\`{a}g in $L^2(D)$, strongly c\'{a}dl\`{a}g in $W^{-1,p'}(D)$ and right-continuous in $L^2(D)$.
\\
The initial value $u(0)=u_0$ is satisfied in this sense.
\end{proposition}
\begin{proof}
A similar argumentation can be found in \cite[Proposition 3.4]{RWZ13}. $\mathds{P}$-a.s. we have for every $t \in [0,T]$
\begin{align*}
u_n(t)- u_0 - \int_0^t \Phi \, dW -\int_0^t \operatorname{div} a(x, u_n, \nabla u_n) \, ds + \int_0^t f(u_n) \, ds = n \int_0^t u_n^- \, ds.
\end{align*}
Testing with $\psi \in \mathcal{D}(D)$ and passing to the limit with $n \to \infty$ yields 
\begin{align}\label{eq 230224_03}
&\lim_{n\to\infty}n \int_D \psi \int_0^t u_n^- \, ds \ dx 
\\=& 
\int_D \psi \left( u(t)- u_0 - \int_0^t \Phi \, dW -\int_0^t \operatorname{div} a(x, u, \nabla u) \, ds + \int_0^t f(u) \, ds\right)dx \nonumber,
\end{align}
and since $n\int_0^t u_n^-ds\geq 0$, it converges in $\mathcal{D}'(D)$ to a non-negative Radon measure on $D$ denoted $\eta_t$. Obviously, for $0\leq s \leq t \leq T$ we have $0 \leq \eta_s \leq \eta_t  \leq \eta_T$ in $\mathcal{D}'(D)$. Especially, $(\eta_t)_{t \in [0,T]}$ is a non-decreasing sequence of Radon measures on $D$. Next, we show that for any $t \in [0,T)$, $t \mapsto \eta_t \in (\mathcal{C}_0(D))'$ has a right limit $\eta_{t^+}$, and, for $t \in (0,T]$, a left limit $\eta_{t^-}$ respectively, in $(\mathcal{C}_0(D))'$. In particular, $\eta_{t^-} \leq \eta_t$ for all $t \in (0,T]$ and $\eta_t \leq \eta_{t^+}$ for all $t \in [0,T)$.
Let $\psi \in \mathcal{C}_0(D)$. Then $\psi = \psi^+ - \psi^-$, where $\psi^+, \psi^- \in \mathcal{C}_0(D)^+$. Hence, $t \mapsto \eta_t(\psi^+)$ and $t \mapsto \eta_t(\psi^-)$ are non-decreasing. Therefore the limits
\begin{align*}
\lim\limits_{s \to t^+} \eta_s(\psi^+), ~~\lim\limits_{s \to t^+} \eta_s(\psi^-)
\end{align*}
as well as
\begin{align*}
\lim\limits_{s \to t^-} \eta_s(\psi^+), ~~\lim\limits_{s \to t^-} \eta_s(\psi^-)
\end{align*}
exist in $\R$, since they are bounded by $\eta_T(\psi^+)$ and $\eta_T(\psi^-)$, respectively. This yields the existence of the limits
\begin{align*}
\lim\limits_{s \to t^+} \eta_s(\psi) &= \lim\limits_{s \to t^+} \big(\eta_s(\psi^+)-\eta_s(\psi^-)\big) =: \eta_{t^+}(\psi), \\
\lim\limits_{s \to t^-} \eta_s(\psi)&= \lim\limits_{s \to t^-} \big(\eta_s(\psi^+)-\eta_s(\psi^-)\big) =: \eta_{t^-}(\psi).
\end{align*}
Moreover, we have
\begin{align*}
\eta_{t^+} &= \lim\limits_{l \to 0^+} \eta_{t+l} \geq \eta_t ~~~\text{and} \\
\eta_{t^-} &= \lim\limits_{l \to 0^+} \eta_{t-l} \leq \eta_t
\end{align*}
in $(\mathcal{C}_0(D))'$. Equality \eqref{eq 230224_03} yields that for $t \in [0,T)$ there exists a right limit and for $t \in (0,T]$ a left limit for $u(t)$ in $(\mathcal{C}_0(D))'$ and we have for any $\psi \in \mathcal{C}_0(D)$
\begin{align}\label{lim1}
\lim_{s \to t^+}\int_D u(s) \psi \, dx = \eta_{t^+}(\psi) - \eta_t(\psi) + \int_D u(t)  \psi \, dx
\end{align}
for any $t \in [0,T)$ and
\begin{align}\label{lim2}
\lim_{s \to t^-}\int_D u(s) \psi \, dx = \eta_{t^-}(\psi) - \eta_t(\psi) + \int_D u(t) \psi \, dx
\end{align}
for any $t \in (0,T]$. Therefore we may define
	\[u_{t^+}:=\eta_{t^+}-\eta_t+u(t), \quad u_{t^-}:=\eta_{t^-}-\eta_t+u(t)\]
in $(\mathcal{C}_0(D))'$.
Let us notice that thanks to Lemma \ref{Lemma 220724_01}, 
\[\omega \mapsto X(\omega):=\sup\limits_{t \in [0,T]} \Vert u(\omega,t) \Vert_{L^2(D)}\] 
is an integrable random variable and, for almost every fixed $\omega$ and all $t\in [0,T]$,  
\[\Vert u(\omega,t) \Vert_{L^2(D)} \leq X(\omega)< \infty.\]
As a consequence of the limit \eqref{lim1}, for fixed $t \in [0,T)$, and any sequence $(s_n)_n$ such that $s_n\geq t$ for all $n\in\mathbb{N}$ converging to $t$,  $(u(s_n))_n$ converges weakly to $u_{t^+}$ in $(\mathcal{C}_0(D))'$. Then, the uniform boundedness in time of $(u(s_n))_n$ in $L^2(D)$ and the density of $\mathcal{C}_0(D)$ in $L^2(D)$ yield $u(s_n)\rightharpoonup u_{t^+}$ in $L^2(D)$ as $n\rightarrow \infty$ and consequently, $\lim_{s\rightarrow t^+} u(s)=u_{t^+}$ weakly in $L^2(D)$. Moreover, we have $u_{t^+} \geq u(t)$ a.e in $D$. Finally, since $\mathds{P}$-a.s. $u \geq 0$ a.e. in $Q_T$, there exists a set $A \subset (0,T)$ of full measure such that $u(t) \geq 0$ a.e. in $D$ for all $t \in A$, where the exceptional set in $D$ may depend on $t\in A$.
Therefore, by selecting a sequence $(s_n)_n\subset A$ converging to $t^+$, one concludes $u_{t^+} \geq 0$ a.e. in $D$. For any $t \in (0,T]$, in a similar way, there exists $u_{t^-}\in L^2(D)$, with  $0 \leq u_{t^-} \leq u(t)$ a.e in $D$, such that $u(s)\rightharpoonup u_{t^-}$ in $L^2(D)$ as $s \to t^-$. Then, using \eqref{4}, more precisely, the difference of that equation at time $s$ to the one at time $t$ for $0 \leq s \leq t \leq T$, one gets, recalling $u_n(0)=u_0$ for all $n\in\mathbb{N}$,
\begin{align*}
&\frac{1}{2} \Vert u_n(t) \Vert_2^2 - \frac{1}{2} \Vert u_n(s) \Vert_2^2 + \int_s^t \int_D a(x, u_n, \nabla u_n) \nabla u_n \, dx \, dr + n \int_s^t \int_D |u_n^-|^2 \, dx \, dr \\
&= -\int_s^t \int_D f(u_n)u_n \, dx \, dr + \int_s^t (u_n , \Phi \, dW)_2 + \frac{1}{2}\int_s^t \Vert \Phi\Vert^2_{\operatorname{HS}} \,dr.
\end{align*}
Passing to the limit with $n \to \infty$ and discarding non-negative terms yield:
\begin{align}\label{eq 230224_04}
	\begin{aligned}
&\frac{1}{2} \Vert u(t) \Vert_2^2 - \frac{1}{2} \Vert u(s) \Vert_2^2\\
& \leq -\int_s^t \int_D f(u)u \, dx \, dr + \int_s^t (u , \Phi \, dW)_2 + \frac{1}{2}\int_s^t \Vert \Phi\Vert^2_{\operatorname{HS}} \,dr +(t-s)\Vert	\kappa\Vert_{L^1(D)},
	\end{aligned}
\end{align}
where we have to replace $u(s)$ by $u_0$ in the case $s=0$.
Using the lower semicontinuity of the norm for the weak convergence, replacing $t$ by $t+l$ and $s$ by $t$, \eqref{eq 230224_04} and $0 \leq u(t) \leq u_{t^+}$ yield, for any $t \in [0,T)$,
\begin{align}\label{eq 260224_01}
\Vert u_{t^+} \Vert_2^2 \leq \liminf\limits_{l \to 0^+} \Vert u(t+l) \Vert_2^2  \leq \limsup\limits_{l \to 0^+} \Vert u(t+l) \Vert_2^2  \leq \Vert u(t) \Vert_2^2  \leq \Vert u_{t^+} \Vert_2^2,
\end{align}
where, again, we have to replace $u(t)$ by $u_0\geq 0$ in the case $t=0$. Since by \eqref{eq 260224_01} it follows for $t \in [0,T)$
\begin{align*}
&\Vert u_{t^+} - u(t) \Vert_2^2 = \Vert u_{t^+} \Vert_2^2 - 2 (u_{t^+}, u(t))_2 + \Vert u(t) \Vert_2^2 \\
&= 2\Vert u(t) \Vert_2^2 - 2 (u_{t^+}, u(t))_2 \leq 2\Vert u(t) \Vert_2^2 - 2 (u(t), u(t))_2 =0,
\end{align*}
one has that $u_{t^+}=u(t)$ and $\lim_{s \to t^+} u(s)=u(t)$ in $L^2(D)$. In particular, $\lim_{s \to 0^+} u(s)=u_0$ in $L^2(D)$.
\end{proof}
\begin{corollary}\label{Corollary 120724_01}
The measure $\eta$ does not charge sets of zero capacity, i.e., for any $E \in \mathcal{B}(Q_T)$ with $\capa_p(E) =0$ we have $\eta(E)=0$.
\end{corollary}
\begin{proof}
Since the result is obvious if $\eta(\omega)=0$, one considers now $\eta(\omega)\neq0$. First, we prove the result for $K \subset Q_T$ compact: Let $K \subset Q_T$ compact with $\capa_p(K)=0$ and $\alpha >0$. Then, by \cite[Proposition 2.14]{DPP} we have CAP$(K)=0$, where CAP is defined as in \cite[Definition 2.10]{DPP}. Now, let $\theta \in \mathcal{C}_c^{\infty}(Q_T)$ such that $\theta \geq \mathds{1}_K$. According to \cite[Remark 2.3]{DPP} there exists $C(\theta)>0$ such that $\Vert \theta \psi \Vert_{\W} \leq C(\theta) \Vert \psi \Vert_{\W}$ for all $\psi \in \W$. By definition of CAP and since CAP$(K)=0$, there exists $\psi \in \mathcal{C}_c^{\infty}([0,T] \times D)$ with $\psi \geq \mathds{1}_K$ such that $\Vert \psi \Vert_{\W} \leq \frac{\alpha}{C(\theta) \Vert \eta \Vert_{\W'}}.$ Then, $\theta \psi \in \mathcal{C}_c^{\infty}(Q_T)$ with $\theta \psi \geq \mathds{1}_K$ and
\begin{equation*}
\eta(K) = \int_{K} 1 \ d\eta = \int_{Q_T} \mathds{1}_K \ d\eta \leq \int_{Q_T} \theta \psi \ d\eta = \langle \eta, \theta \psi \rangle_{\W', \W} \leq \Vert \eta \Vert_{W'} \Vert \theta \psi \Vert_{\W} \leq \alpha.
\end{equation*}
Since $\alpha >0$ was arbitrary chosen, we obtain $\eta(K)=0$. Now let $E \in \mathcal{B}(Q_T)$ with $\capa_p(E)=0$ be arbitrary. Then, for any $K \subset E$ compact we have $\capa_p(K)=0$ and therefore $\eta(K)=0$. Since $\eta$ is a Radon-measure, we have
\begin{equation*}
\eta(E) = \sup \{ \eta(K); ~K \subset E ~\text{compact} \} =0.
\end{equation*}
\end{proof}
\begin{lemma}\label{Lemma 261124_03}
The sequence $(\eta_n^{\varphi})_n$ is equidiffuse, i.e., for all
$\alpha>0$ there exists $\beta>0$ such that for all $E \in \mathcal{B}(Q_T)$ we have 
\[\capa_p(E) < \beta ~\Rightarrow \sup\limits_{n \in \N}\eta_n^{\varphi}(E) < \alpha.\]
\end{lemma}
\begin{proof}
Let $\alpha >0$. Since the result holds obviously if $\sup\limits_{n \in \N}\Vert \eta_n^{\varphi} \Vert_{\mathcal{W}'}=0$, let us assume it is positive and choose $\beta = \frac{\alpha}{3 \sup\limits_{n \in \N}\Vert \eta_n^{\varphi} \Vert_{\mathcal{W}'}}$. Now let $E \in \mathcal{B}(Q_T)$ such that $\capa_p(E) < \beta$. By definition of the parabolic capacity (see Appendix, Section \ref{A}) there exists an open set $U \subset Q_T$ such that $E \subset U$ and $\capa_p(U) < 2\beta$. There exists $\psi \in \mathcal{W}$ such that $\psi \geq \mathds{1}_{U}$ and $\Vert \psi \Vert_{\mathcal{W}} < 3 \beta$. Now, for any $n \in \N$,
\begin{align*}
\eta_n^{\varphi}(E) \leq \eta_n^{\varphi}(U)\leq \int_{Q_T}\psi nu_n^- \varphi \, d(t,x) = \langle \eta_n^{\varphi}, \psi \rangle_{\mathcal{W}', \mathcal{W}} \leq \Vert \eta_n^{\varphi} \Vert_{\mathcal{W}'} \Vert \psi \Vert_{\mathcal{W}} < 3 \beta \Vert \eta \Vert_{\mathcal{W}'} = \alpha.
\end{align*}
\end{proof}
\begin{remark}\label{diffuse}
Therefore, $\eta^\varphi$ satisfies a similar property. Indeed, for $\alpha>0$ and $\beta= \frac{\alpha}{6 \sup\limits_{n \in \N}\Vert \eta_n^{\varphi} \Vert_{\mathcal{W}'}}$  one gets, using the same arguments as in Lemma \ref{Lemma 261124_03}, that for any Borel set $E$ with $\capa_p(E) < \beta$, there exists an open set $U$ such that $E \subset U$ and $\capa_p(U) < 2\beta$. Then \cite[Proposition 4.2.3]{ABM05} yields, $\eta^\varphi(E)\leq \eta^\varphi(U) \leq \liminf_n \eta_n^\varphi(U) \leq \alpha/2 < \alpha.$
\end{remark}

\begin{remark}
		From Corollary \ref{Corollary 120724_01} it follows that $\capa_p$-quasi continuous functions are measurable with respect to $\eta$ and $\eta^{\varphi}$. Moreover, $\capa_p$-quasi continuous representatives coincide quasi everywhere, hence $\eta$-a.e. and therefore also $\eta^{\varphi}$-a.e. (see Remark \ref{Rem4.10}). 
		\end{remark}
\begin{lemma}\label{Lemma 261124_05}
For each $m \in \N$ and $K>0$ we have
\begin{equation*}
\int_{Q_T} T_K(u_m^+) \ d\eta^{\varphi} \leq \liminf\limits_{n \to \infty} \int_{Q_T} T_K(u_m^+) \ d\eta_n^{\varphi},
\end{equation*}
where for each $r \in \R$, $T_K(r) := \max(-K, \min(r,K))$.
\end{lemma}
\begin{proof}
For fixed $m \in \N$ and $K>0$, let us set $v:= T_K(u_m^+)$. We know that $u_m \in \W + \mathcal{C}_{0,D}$, where
\begin{equation*}
\mathcal{C}_{0,D} := \{\psi \in \mathcal{C}([0,T] \times \overline{D}); ~\psi(t,x)=0 ~\text{for all}~(t,x) \in [0,T] \times \partial D\}.
\end{equation*}
Lemma \ref{Lemma 271124_01} yields that there exists a not relabeled extension of $u_m$ to $\overline{Q_T}$ such that $u_m(t,x)=0$ for all $(t,x) \in [0,T] \times \partial D$. For each $l \in \N$ there exists $\widehat{F}_l \subset \overline{Q_T}$ compact such that ${u_m} _{|{\widehat{F}_l}}$ is continuous and $\capa_p(Q_T \setminus \widehat{F}_l) \to 0$ as $l \to \infty$. Since $r \mapsto T_K(r^+)$ is continuous, the same properties apply to $v$. Setting $F_l := \widehat{F}_l \cap Q_T$, we have
\begin{equation*}
\capa_p(Q_T \setminus F_l)= \capa_p(Q_T \setminus \widehat{F}_l)  \to 0 \text{ as } l \to \infty.
\end{equation*}
Consequently, Lemma \ref{Lemma 261124_03} yields that $\sup\limits_{n \in \N} \eta_n^{\varphi}(Q_T \setminus F_l) \to 0$ as $l \to \infty$  and Remark \ref{diffuse}  yields that $\eta^{\varphi}(Q_T \setminus F_l) \to 0$ as $l \to \infty$.\\
Thanks to Tietze-Urysohn theorem (see \cite[Theorem 7.3.8]{M74}) for each $l \in \N$ we can find $v_l \in \mathcal{C}(\overline{Q_T}) \subset \mathcal{C}_b(Q_T)$ such that ${v_l}_{|\widehat{F}_l} = v_{|\widehat{F}_l}$ and $0 \leq v_l \leq K$ on $\overline{Q_T}$. Especially we have ${v_l}_{|F_l} = v_{|F_l}$. Now, using Lemma \ref{narrowconvergence_241011} we obtain
\begin{align*}
\int_{Q_T} v \ d\eta^{\varphi} &= \int_{F_l} v \ d\eta^{\varphi} + \int_{Q_T \setminus F_l} v \ d\eta^{\varphi} \leq \int_{F_l} v_l \ d\eta^{\varphi} +K \cdot \eta^{\varphi}(Q_T \setminus F_l) \\
&\leq \int_{Q_T} v_l \ d\eta^{\varphi} +K \cdot \eta^{\varphi}(Q_T \setminus F_l) \\
&= \lim\limits_{n \to \infty} \int_{Q_T} v_l \ d\eta_n^{\varphi} +K \cdot \eta^{\varphi}(Q_T \setminus F_l) \\
&= \lim\limits_{n \to \infty} \bigg[ \int_{F_l} v_l \ d\eta_n^{\varphi} + \int_{Q_T \setminus F_l} v_l \ d\eta_n^{\varphi} \bigg] +K \cdot \eta^{\varphi}(Q_T \setminus F_l) \\
&\leq \liminf\limits_{n \to \infty} \int_{F_l} v \ d\eta_n^{\varphi} +K \cdot \big( \sup\limits_{n \in \N} \eta_n^{\varphi}(Q_T \setminus F_l) +  \eta^{\varphi}(Q_T \setminus F_l)\big)\\
&\leq \liminf\limits_{n \to \infty} \int_{Q_T} v \ d\eta_n^{\varphi} +K \cdot \big( \sup\limits_{n \in \N} \eta_n^{\varphi}(Q_T \setminus F_l) +  \eta^{\varphi}(Q_T \setminus F_l)\big) \\
\end{align*}
Since $K \cdot \big( \sup\limits_{n \in \N} \eta_n^{\varphi}(Q_T \setminus F_l) + \eta^{\varphi}(Q_T \setminus F_l)\big) \to 0$ as $l \to \infty$, we obtain the required result.
\end{proof}
Recalling that $\bar{u}=\sup\operatorname{quasi \ ess}_{n\in\mathbb{N}}u_n$ 
is a q.e. representative of $u$ and $u\geq 0$ a.e., we may deduce that 
\[\hat{u}:=\bar{u}^+\]
is also a q.e. defined representative of $u$ which is non-negative everywhere in $Q_T$.
\begin{proposition}\label{Proposition 120724_01}
For a set of full measure still denoted by $\widetilde\Omega$, we have $\hat{u}=0$ $\eta$-a.e. in $Q_T$. Especially, we have $ \int_{Q_T} \hat{u} \, d\eta = 0$.
\end{proposition}
\begin{proof}
For $n,m \in \N$ with $n \geq m$, we have $u_n \geq u_m$ q.e. and therefore there exists a set $N \in \mathcal{B}(Q_T)$ not depending on $n$ and $m$ with $\capa_p(N)=0$ such that $\{u_m\geq 0\}\subset \{u_n\geq 0\}\cup N$. Hence $u_m^+ u_n^- = 0$ q.e. in $Q_T$ and therefore $T_K(u_m^+) u_n^- = 0$ a.e. in $Q_T$ for any $K>0$.
Now, for any $K >0$, $T_K(u_m^+) \in L^1(Q_T,\eta^{\varphi})$. Thus using Lemma \ref{Lemma 261124_05} it follows that
\begin{align*}
0\leq \int_{Q_T} T_K(u^+_m)\ d\eta^{\varphi}\leq \liminf_{n\rightarrow\infty}\int_{Q_T} T_K(u^+_m)\varphi nu_n^-\ d(t,x) =0.
\end{align*}
Consequently,
\begin{align*}
\int_{Q_T} T_K(u^+_m)\ d\eta^{\varphi}=0
\end{align*}
for all $K>0$. Now, since $(T_K(u^+_m))_{K>0}$ is a sequence of non-negative, $\eta^{\varphi}$-integrable functions increasing towards $u^+_m$ as $K \to \infty$, it follows that
\begin{align}\label{241011_02}
\int_{Q_T}u^+_m\ d\eta^{\varphi}=\lim_{K \to \infty}\int_{Q_T}T_K(u^+_m)\ d\eta^{\varphi}=0.
\end{align}
Now, we recall that, $\bar{u}=\sup\operatorname{quasi \ ess}_{m\in\mathbb{N}}u_m$, hence, using the fact that $(u^+_m)_m$ is increasing q.e., $\hat{u}=\bar{u}^+=\sup\operatorname{quasi \ ess}_{m\in\mathbb{N}}u_m^+$ and, by monotone convergence theorem, we may pass to the limit for $m \to \infty$ in \eqref{241011_02} to arrive at
\begin{align}\label{241011_03}
\int_{Q_T} \varphi \hat{u} \ d\eta = \int_{Q_T}\hat{u}\ d\eta^{\varphi}=0.
\end{align}
Since $\varphi \hat{u}$ is non-negative q.e. and therefore especially $\eta$-a.e. it now follows that, in $\widetilde\Omega_\varphi$, $\varphi \hat{u} =0$ $\eta$-a.e. in $Q_T$. Now, selecting a sequence $(\varphi_h)_{h\in\mathbb{N}}\subset C_c^{\infty}(Q_T)$ of non-negative functions such that $0\leq \varphi_h\leq 1$ in $Q_T$ and $\varphi_h=1$ in $K_h:=\{z\in Q_T \ : \ \operatorname{dist}(z,\partial Q_T)\geq \frac{1}{h}\}$, updating $\widetilde\Omega$ by the set of full measure $\widetilde\Omega \cap (\cap_h \widetilde\Omega_{\varphi_h})$, we obtain $\hat{u}=0$ $\eta$-a.e. in $Q_T$. As a result we conclude
\begin{equation*}
\int_{Q_T} \hat{u} \ d\eta =0.
\end{equation*}	
\end{proof}

\section{The case  $ \frac{2d}{d+1}<p<2$}\label{section-p-leq-2}
In this section, we consider the following assumptions: 
\begin{itemize}
\item[$(H_1)$] Let $\frac{2d}{d+1}<p<2$ and $\widetilde{f}:\mathbb{R} \to \mathbb{R}$ be a continuous function such that: $\widetilde{f}(0)=0$, $\widetilde{f}$ is non-decreasing, and there exists a constant $C>0$ such that $\vert \widetilde{f}(\lambda)\vert \leq C\vert \lambda \vert^{p-1}.$
\item[$(H_2)$] We assume the integrand $\Phi$ to be progressively measurable in 
\[L^2(\Omega\times (0,T);HS(L^2(D);H^k(D)\cap H^1_0(D)))\] 
for $k\in\mathbb{N}$ such that $k>\max(\frac{d}{2}, \frac{2+d}{2} - \frac{d}{p})$.
\end{itemize}
Now, let us consider \eqref{220822_01} with $\widetilde{f}$ instead of $f$. We want to find a pair $(u,\eta)$ that is a solution in the sense of Definition \ref{Solution} to
\begin{align}\label{pb-reflection-}
	\begin{cases}
		du \ -\operatorname{div}\,(a(\cdot,u,\nabla u))\,dt + \widetilde{f}(u) \, dt= \Phi \,dW +d\eta \\
		u(0,\cdot)=u_0\geq 0, \quad u\geq	0. 
	\end{cases}
\end{align}
With minor changes of our argumentation, we can prove the following result:
\begin{theorem}\label{THM-bis}
	Let the assumptions in Section \ref{Aotsan} be satisfied for $\frac{2d}{d+1}<p<2$ and $\widetilde{f}$ instead of $f$. In addition, let Assumptions $(H_1)$ and $(H_2)$ be satisfied. For $u_0 \in L^2(\Omega \times D)$ $\mathcal{F}_0$-measurable with $u_0 \geq 0$ a.e. in $\Omega \times D$, 
	then there exists a solution $u$ in the sense of Definition \ref{Solution} to \eqref{pb-reflection-}.
\end{theorem}
The proof of Theorem \ref{THM-bis} is based on similar arguments presented above, let us sketch the parts where some minor changes are needed.
\begin{itemize}
\item First, similarly to Remark  \ref{Remark 110724_01},
a slight modification of the results in \cite{VZ21} (see also \cite{VZ2024}) yields that there  exists a unique predictable $u \in L^2(\Omega;\mathcal{C}([0,T]; L^2(D)))\cap L^p(\Omega; L^p(0,T; W_0^{1,p}(D)))$ such that  $\mathds{P}$-a.s.  for all $t \in [0,T]$ the following SPDE
\begin{align}\label{1---}
	u(t)= u_0+\int_0^t \operatorname{div}(a(\cdot, u , \nabla u)) \, ds - \int_0^t  \widetilde{f}(u) \, ds  + \int_0^t \Phi \, dW
\end{align}
holds in $L^2(D)$.
\item Secondly, thanks to the assumptions on $\widetilde{f}$, notice that for any $\lambda_1,\lambda_2\in \mathbb{R}$: \begin{align*}
	(f(\lambda_1) - f(\lambda_2)) \eta_{\delta}'(\lambda_1-\lambda_2) \geq 0,
\end{align*}
therefore, one gets  by using similar arguments to Lemma \ref{Lemma 1} the following lemma.
\begin{lemma}\label{Lemma 1--}
	Let $u$ be a strong solution to \eqref{1---} with a Lipschitz continuous reaction term $\widetilde{f}$ and initial value $u_0 \in L^2(\Omega \times D)$. Moreover, let $v$ be a strong solution to \eqref{1} with a Lipschitz continuous reaction term $\widetilde{g}$ and initial value $v_0\in L^2(\Omega \times D)$, where $\widetilde{f} \geq \widetilde{g}$ and $u_0 \leq v_0$ a.e. in $\Omega \times D$. Then, there exists a measurable set $\Omega_{\widetilde{f},\widetilde{g},u,v}\in\mathcal{F}$ of full measure 
	such that for all $\omega\in \Omega_{\widetilde{f},\widetilde{g},u,v}$: 
\begin{itemize}
\item[$i.)$] for all $t \in [0,T]$,  $u(t) \leq v(t)$ a.e. in $D$;  
\item[$ii.)$] $u \leq v$ quasi everywhere in $Q_T$ (see Section \ref{A}: Appendix).
\end{itemize}
\end{lemma}
\item Next, for $n \in \N$ we set $\widetilde{f}_n(r):= \widetilde{f}(r) - n (r^-)^{p-1}$. Then, there exists a unique predictable process $u_n \in L^2(\Omega; \mathcal{C}([0,T]; L^2(D)))\cap L^p(\Omega; L^p(0,T; W_0^{1,p}(D)))$ such that $u_n(0, \cdot)=u_0$ in $L^2(\Omega\times D)$ and $\mathds{P}$-a.s. in $\Omega$ the following equality
\begin{align*}
	\partial_t (u_n - \int_0^{\cdot} \Phi \, dW) - \operatorname{div}(a(\cdot, u_n , \nabla u_n)) + \widetilde{f}(u_n) - n(u_n^-)^{p-1} =0
\end{align*}
holds  in $L^{p'}(0,T; W^{-1,p'}(D))$.
\item Similar arguments as in Subsection \ref{subsection-estimate-exp} ensure the existence of $\mathbf{K} >0$ not depending on $n \in \N$ such that

\begin{align}\label{est-p2}
	\begin{aligned}
&	\erww{\sup_{t\in[0,T]}\Vert u_n(t)\Vert_2^2}\leq \mathbf{K}, \quad 	\erww{\int_0^T\Vert \nabla u_n(s)\Vert_p^p\,dt}\leq \mathbf{K},\\
&	\erww{\int_0^T\int_D |a(x,u_n,\nabla u_n)|^{p'}\,dx\,dt}\leq \mathbf{K}, \quad 	\erww{\int_0^T\left\Vert n^{\frac{1}{p}} u_n^{-}(s)\right\Vert_p^p\,ds}\leq \mathbf{K}.
\end{aligned}
\end{align}
$\bullet$ By using Lemma \ref{Lemma 1--}, similar results as in Subsection \ref{subsection-pathwise-esti} hold.  Thanks to \eqref{est-p2} the convergence results of Subsection \ref{Subsection-convergence} hold, except changing \eqref{231019_07} by 
\begin{align*}
	u_n^-\rightarrow 0 \ \text{in} \ L^p(\Omega;L^p(0,T;L^p(D))).
\end{align*}
\item With the same arguments as in the proof of the Lemma \ref{lem-cv-reflected}, we get \begin{lemma}
	There exists $\widetilde{\eta}\in L^{p'}(\Omega;\mathcal{W}')$ such that, up to a not relabeled subsequence, 
	\[n(u_n^-)^{p-1}\rightharpoonup \widetilde{\eta}\]
	weakly in $L^{p'}(\Omega;\mathcal{W}')$. Moreover, 
	$(\widetilde{\eta}(\omega))_{\omega\in\Omega}$ is a weak-$\ast$ adapted random non-negative Radon measure.
\end{lemma}
Finally, recall that $L^2(\Omega\times (0,T)\times D) \hookrightarrow L^p(\Omega\times (0,T)\times D) $, thus   $\lim_{n\rightarrow\infty}\widetilde{f}(u_n)=\widetilde{f}(u)$ in  $L^{p^\prime}(\Omega\times (0,T)\times D)$ if $u_n\to u$ in $L^2(\Omega\times (0,T)\times D)$ and by using the above points, it is not difficult to see that the results of Section \ref{section-L1-estimate} and Section \ref{Section-reflected+solution} hold true in this case as well and Theorem \ref{THM-bis} holds.
\end{itemize}
\begin{remark}
It is worth mentioning that studying  deterministic or stochastic  parabolic PDEs with constraints is more challenging in the case of  $\frac{2d}{d+1}<p<2$, due to some regularity in time issues. Hence, we need to consider non Lipschitz penalization with a growth compatible with the one of the main pseudomonotone operator in order to get the correct regularity in time and space to use It\^o formula (integration by parts in the deterministic setting), we refer e.g. to \cite{GuibeMTV2020,Yassine2020,TahraouiVallet2022} for more details.
\end{remark}
\section{Appendix: Parabolic capacity and quasi continuous functions}\label{A}
We recall some basic results from nonlinear capacity theory. In the following, let $Q_T:=(0,T)\times D$,  and the Lebesgue measure on $Q_T$ be denoted by $\lambda$. For $2\leq p<\infty$, the function space
\[\mathcal{W}:=\{v\in L^p(0,T;W^{1,p}_0(D)) \ | \ \partial_t v \in L^{p'}(0,T;W^{-1,p'}(D))\}\]
endowed with its natural norm
\[\Vert v\Vert_{\mathcal{W}}:=\Vert v\Vert_{L^p(0,T;W^{1,p}_0(D))}+\Vert v\Vert_{L^{p'}(0,T;W^{-1,p'}(D))}\]
is a reflexive separable Banach space (see, e.g., \cite[Ch.XVIII, Sec.4]{DL}). Moreover, $\mathcal{W}$ is continuously imbedded in $\mathcal{C}([0,T]; L^2(D))$ (see \cite{DL}) and $\mathcal{C}^{\infty}_c([0,T] \times D)$ is a dense subset of $\mathcal{W}$ (see \cite[Theorem 2.11]{DPP}) with respect to the norm of $\mathcal{W}$.
\begin{definition}[see, {\cite[Definition 2.7]{DPP}}]\label{Definition capacity}
For any open set $U\subset Q_T$ we define the (parabolic $p$-) capacity of $U$ as
\[\capa_p(U)=\inf\left\{\Vert v\Vert_{\mathcal{W}} \ | \ v\in\mathcal{W}, \ v\geq \mathds{1}_U \ \text{a.e. in} \ Q_T\right\}\]
with the convention $\inf\emptyset:=+\infty$. Then, for any Borel subset $B\subset Q_T$ the definition is extended by setting
\[\capa_p(B)=\inf\left\{\capa_p(U) \ | \ U\subset Q_T \ \text{open subset}, \  B\subset U\right\}.\]
\end{definition}
\begin{definition}[see, e.g., {\cite[p. 531]{P}}]\label{Definition cap-quasi continuous}
A function $v:Q_T\rightarrow\mathbb{R}$ is called $\capa_p$-quasi continuous, if for every $\varepsilon>0$ there exists an open set $U_{\varepsilon}\subset Q_T$ with $\capa_p(U_{\varepsilon})\leq \varepsilon$ and such that $v_{|Q_T\setminus U_{\varepsilon}}$ is continuous on $Q_T\setminus U_{\varepsilon}$.
\end{definition}
We recall that $\lambda(E)^{1/p} \leq \capa_p(E)$ for all $E \in \mathcal{B}(Q_T)$ and therefore any set of zero capacity is $\lambda$-negligible. The converse is, in general, not true. Therefore the sets of zero capacity may be considered as exceptional sets that are finer than $\lambda$-negligible
sets. A property will be said to hold $\capa_p$-quasi everywhere, i.e., "q.e.", if it holds everywhere except on a set of zero capacity.
\begin{lemma}[see, e.g., {\cite[Lemma 2.20]{DPP}}]\label{220822_lem01}
Any element $v$ of $\mathcal{W}$ has a $\capa_p$-quasi continuous representative $\widetilde{v}$ which is $\capa_p$-quasi everywhere unique in the sense that two $\capa_p$-quasi continuous representatives of $v$ are equal except on a set of zero capacity.
\end{lemma}
The uniqueness part of Lemma \ref{220822_lem01} is a direct consequence of the following (stronger) result:

\begin{lemma}\label{LemPositiveQuasiCont}
If $u:Q_T\rightarrow \mathbb{R}$ is a $\capa_p$-quasi continuous function such that $u \geq 0$ a.e. in $Q_T$, then $u \geq 0$ q.e. in $Q_T$.
\end{lemma}
\begin{proof}
Denote by $A=\{u<0\}$ and  since $u$ is $\capa_p$-quasi continuous, there exists a sequence of open sets $(U_j)_{j\in\mathbb{N}}\subset Q_T$ such that $\capa_p(U_j)=\frac{1}{j}$ and $u$ is continuous on $U_j^C:=Q_T\setminus U_j$ for all $j\in\mathbb{N}$. As a consequence, $A \cup U_j$ is a $\capa_p$-quasi open set and there exists a sequence $(v_j) \subset \mathcal{W}$ such that $\|v_j\|_{\mathcal{W}} \to 0$ and $v_j \geq \mathds{1}_{U_j}$. Since $A$ is negligible for the Lebesgue measure on $Q_T$, $A \cup U_j=U_j$ a.e. and $v_j \geq \mathds{1}_{A \cup U_j}$ a.e.. Thus, $\capa_p(A) \leq \capa_p(A \cup U_j) \leq \|v_j\|_{\mathcal{W}} \to 0$ and therefore $\capa_p(A)=0$ (see \cite[Lemma 2.1]{PPP11}).
\end{proof}
\begin{corollary}\label{CorQuasiEqual}
If $u,v:Q_T\rightarrow\mathbb{R}$ are $\capa_p$-quasi continuous and $u=v$ a.e. in $Q_T$, then $u=v$ q.e. in $Q_T$.
\end{corollary}
\begin{proof} 
This is a direct consequence of the above lemma noticing that $u=v$ if and only if $u-v\geq 0$ and $v-u \geq 0$.
\end{proof}
\begin{lemma}\label{Lemma 261124_04}
Let $\tilde{Q}_T:= (-T,2T) \times D$ and denote by $\widetilde{\capa}_p$ the p-parabolic capacity with respect to $\tilde{Q}_T$, i.e., in Definition \ref{Definition capacity} replace $\mathcal{W}$ by
\begin{equation*}
\mathcal{W}(-T,2T):= \{v\in L^p(-T,2T;W^{1,p}_0(D)) \ | \ \partial_t v \in L^{p'}(-T,2T;W^{-1,p'}(D))\}.
\end{equation*}
Then, for any $E \in \mathcal{B}(Q_T)$ we have
\begin{equation*}
\capa_p(E) \leq \widetilde{\capa}_p(E) \leq 3 \capa_p(E).
\end{equation*}
\end{lemma}
\begin{proof}
Let $\varepsilon >0$. Then there exists $U \subset Q_T$ open with $E \subset U$ and $\capa_p(E) \leq \capa_p(U) \leq \capa_p(E) + \eps$. Now let $v \in \mathcal{W}$ with $v \geq \mathds{1}_U$ a.e. in $Q_T$ and $\capa_p(U) \leq \Vert v \Vert_{\W} \leq \capa_p(U) + \eps$. Especially, we have
\begin{equation*}
\capa_p(E) \leq \Vert v \Vert_{\W} \leq \capa_p(E) + 2\eps.
\end{equation*}
Now set
\begin{equation*}
\hat{v}(t):= \begin{cases} v(-t), ~t \in (-T,0),\\
v(t), ~t \in (0,T),\\
v(2T-t), ~t \in (T, 2T).
\end{cases}
\end{equation*}
Then, $\hat{v} \in \W(-T,2T)$ (see proof of \cite[Theorem 2.11]{DPP}), $\hat{v} \geq \mathds{1}_U$ a.e. in $Q_T$ and $\Vert \hat{v} \Vert_{\W(-T,2T)} \leq 3 \Vert v \Vert_{\W}$. Hence
\begin{equation*}
\Vert \hat{v} \Vert_{\W(-T,2T)}  \leq 3 \Vert v \Vert_{\W} \leq 3\capa_p(E) + 6\eps
\end{equation*}
and therefore $\widetilde{\capa}_p(E) \leq 3 \capa_p(E)$. Now let $w \in \W(-T,2T)$. Then we set $\hat{w} := w_{|Q_T} \in \W$ and we have
\begin{align*}
\Vert \hat{w} \Vert_{\W} &= \Vert \hat{w} \Vert_{L^p(0,T;W_0^{1,p}(D))} + \Vert \partial_t \hat{w} \Vert_{L^{p'}(0,T;W^{-1,p'}(D))} \\
&= \bigg( \int_0^T \Vert \hat{w} \Vert_{W_0^{1,p}(D)}^p \ dt\bigg)^{\frac{1}{p}} + \bigg( \int_0^T \Vert \partial_t \hat{w} \Vert_{W^{-1,p'}(D)}^{p'} \ dt\bigg)^{\frac{1}{p'}} \\
&\leq \bigg( \int_{-T}^{2T} \Vert w \Vert_{W_0^{1,p}(D)}^p \ dt\bigg)^{\frac{1}{p}} + \bigg( \int_{-T}^{2T} \Vert \partial_t w \Vert_{W^{-1,p'}(D)}^{p'} \ dt\bigg)^{\frac{1}{p'}} \\
&= \Vert w \Vert_{\W(-T,2T)}.
\end{align*}
Now, let $E \in \mathcal{B}(Q_T)$ and $\eps >0$. Then, similar as before, there exists $U \subset \tilde{Q}_T$ open such that $E \subset U$ as well as $w \in \W(-T,2T)$ with $w \geq \mathds{1}_U$ a.e. in $\tilde{Q}_T$ such that
\begin{equation*}
\widetilde{\capa}_p(E) \leq \Vert w \Vert_{\W(-T,2T)} \leq \widetilde{\capa}_p(E) + 2\eps.
\end{equation*}
We set $\hat{w} := w_{|Q_T}$. Then $\hat{w} \in \W$ and $\hat{w} \geq \mathds{1}_{U\cap Q_T}$ a.e. in $Q_T$ and
\begin{equation*}
\capa_p(E) \leq\Vert \hat{w} \Vert_{\W} \leq \Vert w \Vert_{\W(-T,2T)} \leq \widetilde{\capa}_p(E) + 2\eps.
\end{equation*}
Hence, we have $\capa_p(E) \leq \widetilde{\capa}_p(E)$.
\end{proof}
\begin{lemma}\label{Lemma 271124_01}
Let $v \in \W + \mathcal{C}_{0,D}$, where
\begin{equation*}
\mathcal{C}_{0,D} := \{\psi \in \mathcal{C}([0,T] \times \overline{D}); ~\psi(t,x)=0 ~\text{for all}~(t,x) \in [0,T] \times \partial D\}.
\end{equation*}
Then, there exists an extension $\bar{v}$ of the $\capa_p$-quasi continuous representative of $v$ to $\overline{Q_T}$ such that $\bar{v}(t,x)=0$ for all $(t,x) \in [0,T] \times \partial D$ and for each $l \in \N$ there exists $\widehat{F}_l \subset \overline{Q_T}$ compact such that $\bar{v}_{|\widehat{F}_l}$ is continuous and $\capa_p(Q_T \setminus \widehat{F}_l) \to 0$ as $l \to \infty$.
\end{lemma}
\begin{proof}
For $v \in \W + \mathcal{C}_{0,D}$ we write $v=v_1 + v_2$, where $v_1 \in \W$ and $v_2 \in \mathcal{C}_{0,D}$. Since $v_2$ is already defined on $\overline{Q_T}$ satisfying $v_2(t,x)=0$ for all $(t,x) \in [0,T] \times \partial D$, there is nothing to show for $v_2$. Hence, w.l.o.g we can assume $v \in \W$. Let $\hat{v}$ be the extension of $v$ to $\tilde{Q}_T = (-T, 2T) \times D$ as in the proof of Lemma \ref{Lemma 261124_04}. 
Then $\hat{v} \in \W(-T,2T)$. 
Recalling the proof of Lemma 2.20 in \cite{DPP}, there exists a sequence $(\tilde{v}^m)_m \subset \mathcal{C}_c^{\infty}([-T,2T] \times D)$ and a representative $\tilde{v}$ of $\hat{v}$ such that $\tilde{v}^m \to \hat{v}$ in $\W(-T,2T)$ and for each $l \in \N$ there exists a relatively closed set $\tilde{F}_l \subset \tilde{Q}_T$ such that $\tilde{v}^m \to \tilde{v}$ uniformly on $\tilde{F}_l$ and $\widetilde{\capa}_p(\tilde{Q}_T \setminus \tilde{F}_l) \to 0$ as $l \to \infty$. Especially, $\tilde{v}$ is a $\capa_p$-quasi continuous representative of $\hat{v}$ in $\tilde{Q}_T$. Now, we extend $\tilde{v}$ to $\tilde{Q}_T \cup (-T,2T) \times \partial D$ by setting $\tilde{v}(t,x) = 0 $ for all $(t,x) \in (-T,2T) \times \partial D$. 
Then we claim that $\tilde{v}_{| \hat{F}_l}$ is continuous, where $\hat{F}_l := \tilde{F}_l \cup (-T,2T) \times \partial D$. In fact, let $(t,x) \in \hat{F}_l$ and $(t_k,x_k)_k \subset \hat{F}_l$ such that $(t_k,x_k) \to (t,x)$. In the case $(t,x) \in \tilde{F}_l$ there is nothing to show since for all but finite $k \in \N$ we necessarily have $(t_k,x_k) \in \tilde{F}_l$. If $(t,x) \in (-T,2T) \times \partial D$, w.l.o.g we can assume that $(t_k,x_k) \in \tilde{F}_l$ for any $k \in \N$. We have
\begin{align}\label{eq 271124_01}
\begin{aligned}
|\tilde{v}(t,x) - \tilde{v}(t_k,x_k)| &\leq |\tilde{v}(t,x) - \tilde{v}^m(t,x)| + |\tilde{v}^m(t,x) - \tilde{v}^m(t_k,x_k)| + |\tilde{v}^m(t_k,x_k) - \tilde{v}(t_k,x_k)| \\
&= |\tilde{v}^m(t,x) - \tilde{v}^m(t_k,x_k)| + |\tilde{v}^m(t_k,x_k) - \tilde{v}(t_k,x_k)|
\end{aligned}
\end{align}
Let $\eps >0$. Since $\tilde{v}^m \to \tilde{v}$ uniformly on $\tilde{F}_l$, there exists $m_0 \in \N$ such that for all $m \geq m_0$:
\begin{equation*}
\sup\limits_{(s,y) \in \tilde{F}_l}|\tilde{v}^m(s,y) - \tilde{v}(s,y)| < \frac{\eps}{2}.
\end{equation*}
For $m=m_0$, there exists $k_0 \in \N$ such that for all $k \geq k_0$:
\begin{equation*}
|\tilde{v}^m(t,x) - \tilde{v}^m(t_k,x_k)| < \frac{\eps}{2}.
\end{equation*}
Those estimates together with \eqref{eq 271124_01} yield that for $k \geq k_0$
\begin{equation*}
|\tilde{v}(t,x) - \tilde{v}(t_k,x_k)|< \frac{\eps}{2} + \frac{\eps}{2} = \eps.
\end{equation*}
Now, setting $\widehat{F}_l = \hat{F}_l \cap \overline{Q_T}$, we can see that $\widehat{F}_l$ is a compact subset of $\R^{d+1}$ with $\widehat{F}_l \subset \overline{Q_T}$.\\
In fact, since $\tilde{F}_l$ is relatively closed in $\tilde{Q}_T$, there exists $\tilde{E}_l \subset \R^{d+1}$ closed such that $\tilde{F}_l = \tilde{E}_l \cap \tilde{Q}_T$. Therefore we have $\hat{F}_l = \big( \tilde{E}_l \cap \tilde{Q}_T \big) \cup (-T,2T) \times \partial D$ and
\begin{align*}
\widehat{F}_l &= \bigg( \big( \tilde{E}_l \cap \tilde{Q}_T \big) \cup (-T,2T) \times \partial D\bigg) \cap \overline{Q_T} \\
&=  \big( \tilde{E}_l \cap \tilde{Q}_T \cap \overline{Q_T} \big) \cup \big( [0,T] \times \partial D \big) \\
&= \big( \tilde{E}_l \cap [0,T] \times D \big) \cup \big( [0,T] \times \partial D \big) \\
&= \big( \tilde{E}_l \cup [0,T] \times \partial D \big) \cap \overline{Q_T}.
\end{align*}
Now, Lemma \ref{Lemma 261124_04} yields
\begin{equation*}
\capa_p(Q_T \setminus \widehat{F}_l) \leq \widetilde{\capa}_p(Q_T \setminus \widehat{F}_l) \leq \widetilde{\capa}_p(\tilde{Q}_T \setminus \widehat{F}_l) \to 0 \text{ as } l \to \infty.
\end{equation*}
Setting $\bar{v}:= \tilde{v}_{|\overline{Q_T}}$ gives the result.
\end{proof}
\begin{lemma}\label{241211_lem01}
Let $v \in \W + \mathcal{C}_{0,D}$. Then, the extension $\bar{v}$ to $\overline{Q_T}$ as in Lemma \ref{Lemma 271124_01} coincides with the representative of $v$ in $\mathcal{C}([0,T];L^2(D))$. Especially, in this sense, $v(0, \cdot)$ and $v(T, \cdot)$ are well defined a.e. in $D$ for the $\capa_p$-quasi continuous representative of $v$.
\end{lemma}
\begin{proof}
Let $\hat{v}$ be as in the proof of Lemma \ref{Lemma 271124_01}, $\tilde{v}$ the representative in $\mathcal{C}([-T, 2T];L^2(D))$ and $(\tilde{v}^m)_m \subset \mathcal{C}_c([-T,2T] \times D)$ such that $\tilde{v}^m \to \tilde{v}$ in $\mathcal{C}([-T, 2T];L^2(D))$ as well as $\tilde{v}^m \to \hat{v}$ $\widetilde{\capa_p}$-q.e. in $(-T, 2T) \times D$. Now, let $t \in [0,T]$. Then, since $\widetilde{\capa_p}(\{t\} \times A)=0$ if and only if $A\subset D$ is Borel measurable with measure zero, we obtain $\tilde{v}^m(t, \cdot) \to \hat{v}(t, \cdot) = \bar{v}(t, \cdot)$ a.e. in $D$. On the other hand we have $\tilde{v}^m(t) \to \tilde{v}(t)$ in $L^2(D)$, hence this convergence holds true a.e. in $D$ for a subsequence. Therefore $\bar{v}(t, \cdot) = \tilde{v}(t, \cdot)$ a.e. in $D$. Since $\tilde{v}_{|[0,T]}$ is a representative of $v$ in $\mathcal{C}([0,T]; L^2(D))$, this yields the result.
\end{proof}
\textbf{Acknowledgments} This work has been partly supported by the German Research Foundation project (ZI 1542/3-1) and the Procope programme for Project-Related Personal Exchange (France-Germany). The research of Yassine Tahraoui is funded by the European Union (ERC, NoisyFluid, No. 101053472). Views and opinions expressed are however those of the author only and do not necessarily reflect those of the European Union or the European Research Council. Neither the European Union nor the granting authority can be held responsible for them.
\bibliographystyle{plain}
\bibliography{obstacle-additive.bib}

\end{document}